\documentclass{daj}
\usepackage{fullpage,  amssymb, amsthm, amsmath, amsfonts}
\usepackage{enumitem}
\usepackage{mathrsfs}
\usepackage{mathtools}
\usepackage{centernot}
\usepackage{xfrac}
\usepackage{eufrak}
\usepackage{stmaryrd}
\usepackage{bm}

\theoremstyle{plain} %italicizes text
\newtheorem{theorem}{Theorem}[section]
\newtheorem*{theorem*}{Theorem}
\newtheorem{prop}[theorem]{Proposition}

\newtheorem{lemma}[theorem]{Lemma}
\newtheorem{cor}[theorem]{Corollary}

\theoremstyle{remark}
\newtheorem{exr}{Exercise}
\newtheorem*{rmk}{Remark}

\theoremstyle{definition}
\newtheorem{defn}{Definition}

\renewcommand{\mod}[1]{{\ifmmode\text{\rm\ (mod~$#1$)}\else\discretionary{}{}{\hbox{ }}\rm(mod~$#1$)\fi}}

\renewcommand{\bar}{\overline}

\newsymbol\dnd 232D

\renewcommand{\a}{\alpha}
\renewcommand{\b}{\beta}

\newcommand{\Z}{\mathbb Z}

\newcommand{\Q}{\mathbb Q}
\newcommand{\N}{\mathbb N}
\newcommand{\R}{\mathbb R}

\newcommand{\p}{\mathcal P}

\renewcommand{\c}{\mathcal C}
\newcommand{\h}{\mathcal H}

\newcommand{\s}{\mathcal S}

\renewcommand{\l}{\lambda}
\newcommand{\E}{\mathcal E}

\newcommand{\vol}{\text{vol}}

\newcommand{\Span}{\text{span}}

\newcommand{\ignore}[1]{}

\dajAUTHORdetails{%
  title = {Khovanskii's Theorem and Effective Results on Sumset Structure}, %% please capitalize all significant words
  author = {Michael J. Curran and Leo Goldmakher},
  plaintextauthor = {Michael J. Curran and Leo Goldmakher},
  plaintexttitle = {Khovanskii's Theorem and Effective Results on Sumset Structure}, %%  title without math or LaTeX
  runningtitle = {Effective Khovanskii},
  runningauthor = {Michael J. Curran and Leo Goldmakher},
  copyrightauthor = {Michael J. Curran and Leo Goldmakher},
  keywords = {Ehrhart theory, iterated sumsets},
}   %%% END \dajAUTHORdetails

\dajEDITORdetails{%
   year={2021},
   number={27},
   received={30 November 2020},   % received date: example: 7 January 2017
   revised={3 November 2021},    % Optional revised date (you may comment it out)
   published={23 December 2021},  % published date
   doi={10.19086/da.28814},       % XXX = number of paper, e.g. da006 for paper#6
}   %%% END \dajEDITORdetails

\begin{document}

\begin{frontmatter}[classification=text]
%% EDITOR: this will force the keywords to appear right after the Abstract.
%%   If the abstract is too long and would force the keywords off the
%%   front page, please comment out % [classification=text] above
%%   This way the keywords will be floated on the bottom of the first page
%%   even though the Abstract spills over to the next page.

%%% AUTHOR: Title goes here.  This line is optional.  You must use it
%%   if title has footnote attached or requires nontrivial typesetting,
%%   e.g., inclusion of linebreaks to force nice layout.
\title{Khovanskii's Theorem and Effective Results on Sumset Structure} %% please capitalize all significant words

%%% AUTHOR:
%%% List all authors. If you wish, place grant acknowledgements in \thanks.
%%% In brackets include a short tag for each author.
\author[mjc]{Michael J. Curran}
\author[lg]{Leo Goldmakher}

%%% AUTHOR: Abstract goes here
\begin{abstract}
A remarkable theorem due to Khovanskii asserts that for any finite subset $A$ of an abelian group, the cardinality of the $h$-fold sumset $hA$ grows like a polynomial for all sufficiently large $h$. Currently, neither the polynomial nor what sufficiently large means are understood.
In this paper we obtain an effective version of Khovanskii's theorem
for any $A \subset \mathbb{Z}^d$ whose convex hull is a simplex; previously, such results were only available for $d=1$.
Our approach gives information about not just the cardinality of $hA$, but also its structure, and we prove two effective theorems describing $hA$ as a set: one answering a recent question posed by Granville and Shakan, the other a Brion-type formula that provides a compact description of $hA$ for all large $h$.
As a further illustration of our approach, we derive a completely explicit formula for $|hA|$ whenever $A \subset \Z^d$ consists of $d+2$ points.
\end{abstract}
\end{frontmatter}

\section{Introduction}

Given a finite set $A \subset \Z^d$, a central object of study in arithmetic combinatorics is the $h$-fold sumset
\[
hA := \{\bm{x}_1 + \cdots + \bm{x}_h : \bm{x}_i \in A\} .
\]
Both the structure and the cardinality of sumsets can be quite complicated, but Khovanskii made the beautiful discovery that once enough copies of $A$ are added together, the behavior stabilizes:
\begin{theorem}[Khovanskii \cite{KhoProof}] \label{thm:KhoMainThm}
Given a finite set $A \subset \Z^d$, there exists a polynomial $p \in \Q[x]$ of degree at most $d$ such that
\(
|hA| = p(h)
\)
for all sufficiently large $h$.
Moreover, if the difference set $A-A$ generates all of $\Z^d$ additively, then $\deg p = d$ and the leading coefficient of $p$ is the volume of the convex hull of $A$.
\end{theorem}

Khovanskii's original proof interprets $|hA|$ as the Hilbert function of a finitely generated graded module over the ring of polynomials in several variables
and then employs the Hilbert polynomial theorem.
This approach is elegant but ineffective: it yields no information about $p(h)$ apart from its degree and leading term, nor any indication of where the phase transition occurs (i.e.\ what ``sufficiently large'' means).
There have been other proofs of Khovanskii's theorem since, including
a geometric proof (which also patches an error in Khovanskii's original paper) by Lee \cite{LeeKhoFix} and a purely combinatorial proof by Nathanson and Ruzsa \cite{NathansonRuzsa,Ruzsa}, but to our knowledge no effective version of Khovanskii's theorem is known for subsets of $\Z^d$ for any $d > 1$.
In this paper we give a different approach to Khovanskii's theorem
that yields more information than previous approaches about the structure of the polynomial and where the phase transition occurs. In some cases, our approach produces a complete description of the cardinality of $hA$ for \emph{all} $h$.

The special case $A \subset \Z$ has received a fair bit of attention (see e.g.\ \cite{GranvilleShakan, GranvilleWalker, Nathanson, WCC}), sometimes under the name of the Frobenius coin problem or the chicken nugget problem.
By shifting and dilating $A$, we may assume that its minimal element is 0 and that the greatest common divisor of its elements is 1. It follows that
\[
\bigcup_{h\geq 0} hA = \N \setminus \E(A)
\]
for some finite {exceptional set} $\E(A)$.\footnote{Here and throughout we define $\N$ to be the set of \emph{non-negative} integers.}
Very recently, Granville and Walker \cite[Theorem 1]{GranvilleWalker} proved that if $b$ is the largest element of $A$, then for any $h \geq b - |A| + 2$ we have
\begin{equation} \label{eqn:GS1DStrucutre}
hA =
\{0,1,\ldots, bh\} \setminus
\bigg(\E(A) \cup \Big(bh - \E(b - A)\Big)\bigg),
\end{equation}
and moreover that the bound $h \geq b - |A| + 2$ is sharp.
This result on the structure of $hA$ can be used to produce a more explicit version of Khovanskii's theorem for subsets of $\Z$. For example, suppose $A = \{0,a,b\}$ where $0 < a < b$ and $(a,b) = 1$.
Classical work of Sylvester \cite{Sylvester} implies that
\[
|\E(A)| = \frac{1}{2}(a-1)(b-1) ,
\]
It is also easy to see that $\E(A) \subseteq [0,ab)$ since the numbers $0,a,2a, \cdots , (b-1)a$ form a complete residue set modulo $b$, hence also that $bh - \E(b - A) \subset (bh - (b-a - 1) b,bh]$.
These facts in combination with \eqref{eqn:GS1DStrucutre} yield
\[
\phantom{\qquad \forall h \geq 2 \lfloor \frac{b}{2}\rfloor .}
|h A| = bh - \frac{1}{2}b^2 + \frac{3}{2}b
\qquad \forall h \geq b
\]
since the sets $[0,ab)$ and $(bh - (b-a) b,bh]$ are disjoint fot $h \geq b$.
This leaves open the question of whether $b $ is the true location of the phase transition, as well as what the behavior of $|hA|$ is for small values of $h$.

The approach we introduce in the present work allows us to completely resolve this question: we will show that
\[
|h A| =
\begin{cases}
\frac{1}{2} h^2 + \frac{3}{2} h + 1 & \mbox{if } 0 \leq h < b - 2 \\
bh - \frac{1}{2}b^2 + \frac{3}{2}b & \mbox{if } h \geq b - 2.
\end{cases}
\]
The proof of this is in fact very short, and can be found at the beginning of section \ref{sec:ExplicitFormula}.
Moreover, we can generalize this to arbitrary dimension and describe the growth of $hA$ for any $A \subset \Z^d$ containing $d+2$ elements:

\begin{theorem}\label{thm:dplus2}
Suppose
$A \subset \Z^d$ consists of $d+2$ elements, and further that $A-A$ generates $\Z^d$ additively.
Let $\Delta_A$ denote the convex hull of $A$.
Then
\[
|hA| = \binom{h + d + 1}{d + 1}
 \qquad \text{whenever } 0 \leq h < \vol(\Delta_A) \cdot d! - d - 1
\]
and
\[
|hA| = \binom{h + d + 1}{d + 1} - \binom{h-\vol(\Delta_A) \cdot d!+d+1}{d+1}
\qquad \text{whenever } h \geq \vol(\Delta_A) \cdot d! - d - 1 .
\]
\end{theorem}
\begin{rmk}
One counterintuitive consequence of this is that for small $h$, the cardinality of $hA$ is independent of the specific elements of $A$. This is because for small values of $h$ each element in $hA$ has a unique representation as a sum of elements of $A$.
\end{rmk}

For a general set $A \subset \Z^d$ with $d > 1$, the structure of the sumset of $hA$ is less well understood. Granville and Shakan \cite{GranvilleShakan}
recently proved a higher dimensional but ineffective analogue of (\ref{eqn:GS1DStrucutre}), and asked for an explicit bound on the phase transition. We are able to deduce such a bound in the case that the convex hull of $A$ is a $d$-simplex:

\begin{theorem}\label{thm:SimplexStructure}
Let $A \subset \Z^d$ be a finite set such that $A-A$ generates $\Z^d$ additively and $\Delta_A$
is a $d$-dimensional simplex.
Denote by $\bm{v}_1,\ldots, \bm{v}_{d+1}$  the vertices of $\Delta_A$ and let
\[
T_i(A) = \bigcup_{k\geq 0} k(A - \bm{v}_i).
\]
Then
for all non-negative integers $h\geq  \vol(\Delta_A)\cdot (d+1)! - 2 - 2d$ we have
\begin{equation}\label{eq:thm1.3}
    hA = \bigcap_{i = 1}^{d+1} \bigg(h\bm{v}_i + T_i(A) \bigg)
\end{equation}

\end{theorem}

\begin{rmk}
Note that the $T_i(A)$ are independent of $h$, so the only dependence on $h$ in the right hand side of (\ref{eq:thm1.3}) lies in the dilates $h \bm{v}_i$.
\end{rmk}

Theorem \ref{thm:SimplexStructure} gives an expression for $hA$ but can be difficult to use in practice.
It turns out that by translating the problem into the language of power series, one can describe the elements of $hA$ more explicitly.
To any set $A \subseteq \Z$, associate the power series $\sum\limits_{a \in A} x^a$; for example, $A = \{0,2,5\}$ would correspond to $1 + x^2 + x^5$.
For this choice of $A$, we will show that for all $h \geq 3$ the power series associated to $hA$ is given by
\[
\frac{1 + x^2 + x^4 + x^6 + x^8 - x^{5 h-7} (1 + x^3 + x^6 + x^9 + x^{12})}{1-x^5} .
\]
This may appear complicated at first glance, but for large values of $h$ it produces a compact description of the set $hA$.
In Theorem \ref{thm:SumsetGFLocalRepThm} we generalize this phenomenon, proving that for any $A$ the power series associated to $hA$ is the ratio of two explicit (and easy to compute) polynomials associated to $A$.
This is analogous to a famous formula of Brion \cite{Brion} expressing the lattice generating function of a convex polytope in terms of the lattice generating functions of its tangent cones.

If rather than associating a power series to $hA$ in the manner described above one studies the standard generating function of $|hA|$, it's possible to obtain an effective version of Khovanskii's theorem for simplicial sumsets, i.e.\ those $A$ whose convex hull is a simplex:

\begin{theorem}\label{thm:SimplexSize}
If $A \subset \Z^d$ is a finite set such that
$A-A$ generates $\Z^d$ additively and
$\Delta_A$ is a $d$-dimensional simplex,
then there exists a polynomial $p \in \Q[x]$
such that $|hA| = p(h)$ for all non-negative $h \geq \vol(\Delta_A)\cdot (d+1)! - 1 - 3d$.
\end{theorem}

The key new idea that allows us to prove all our results on iterated sumsets is that rather than studying the structure of $hA$ individually for each $h$, we embed them all into a higher-dimensional space and study the geometry of the resulting object (called a \emph{cone}).
This idea is essentially a geometric version of a generating function, and is inspired by work of Ehrhart \cite{EhrhartPaper} on counting lattice points in dilates of polytopes. More precisely, Ehrhart used this approach to prove that for any convex polytope $\p \subset \R^d$ whose vertices are lattice points, there exists a polynomial $p\in \Q[t]$ such that the number of lattice points in the $t^\text{th}$ dilate of $\p$ is precisely $p(t)$ for all $t\in \N$ (see \cite{EhrhartTheory} for more background on Ehrhart theory, including a proof of this theorem).
A key difference between our proof and the proof of Ehrhart's theorem is that for sumsets the associated cone is not \emph{simplicial}, meaning that the cardinality of its minimal generating set is greater than its dimension.
It is this difference that causes difficulty in obtaining information on the phase transition when $\Delta_A$ is not a simplex.

We are not the first to connect Khovanskii's theorem to Ehrhart theory; in 2008, Jel\'{\i}nek and Klazar \cite{JelKla} proved a common generalization of Khovanskii's theorem and Ehrhart's theorem.
In their work Jel\'{\i}nek and Klazar employ Dickson's lemma to show that a certain set has finitely many minimal elements,
a tool which is also used in Nathanson and Ruzsa's combinatorial proof of Khovanskii's theorem in \cite{NathansonRuzsa,Ruzsa}.
While very clean, this has the disadvantage of rendering their results ineffective. Indeed, not only does Jel\'{\i}nek and Klazar's main theorem not yield an effective version of Khovanskii's theorem, it only implies an ineffective version of Ehrhart's theorem (the original version of which is effective).

Before concluding this introduction we briefly discuss the interesting work of Barvinok and Woods \cite{BarvinokWoods}, in which rather than looking at sumsets they investigate lattice generating functions for linear transformations of rational polytopes.
They study the complexity of computing such generating functions, in particular showing that there exist polynomial-time algorithms for accomplishing this.
Phrased in terms of sumsets, Barvinok and Woods bound $|hA|$ in terms of the heights of the generators of the cone generated by $A$.
However, since they bound neither the number of these generators nor their heights, their results are of necessity ineffective.
One of the key innovations in our work is an explicit bound on the heights of the generators in terms of the geometry of $A$ (see section \ref{sec:GenSeries} below, in particular Lemmas \ref{lem:minheightbound} and \ref{lem:genseriesdegbound}), which is what allows us to prove effective versions of Khovanskii's theorem. Moreover, we derive a structure theorem and a Brion-like formula for $hA$.
It would be interesting to obtain analogues of our results in the more general setting of Barvinok-Woods.

The structure of this paper is as follows.
In section \ref{sec:ExplicitFormula} we illustrate our approach using some explicit examples; generalizing these, we deduce Theorem \ref{thm:dplus2} in the special case that the convex hull of $A$ is a simplex.
Next, in section \ref{sec:GenSeries}, we use our approach to prove Theorem \ref{thm:SimplexSize}, an effective version of Khovanskii's theorem that holds for all sets $A$ whose convex hull is a simplex.
In section \ref{sec:structureThm} we build on these ideas to prove Theorem \ref{thm:SimplexStructure}, an effective structure theorem on iterated sumsets.
We explore the structure of $hA$ further in section \ref{sec:BrionFormula} and obtain an explicit and compact Brion-type formula capturing the structure of $hA$ for any $A \subset \Z$.
In section \ref{sec:nonsimplex} we return to Theorem \ref{thm:dplus2} and prove it (i.e.\ we remove the additional hypothesis we made in section \ref{sec:ExplicitFormula}).
We conclude with section \ref{sec:FurtherDirections}, which contains a few conjectures and empirical observations that we hope will inspire further research.

\section{Warm up: Explicit Formulae for $|hA|$} \label{sec:ExplicitFormula}

To illustrate our approach, we start by computing $|hA|$ for some simple sets ${A = \{\bm{a}_1, \ldots, \bm{a}_k\} \subset \Z^d}$. Throughout this section, we'll assume that the convex hull of $A$ is a simplex, and that $A$ contains the origin and generates $\Z^d$ additively.

Our primary object of study will be the \emph{cone over $A$}, a $(d+1)$-dimensional object that captures the structure of $hA$ for all $h$ simultaneously. To define this precisely, we first need a bit of notation.
Given $\bm{a} = (a_1,\ldots, a_d) \in \Z^d$, define its \emph{lift} $\widetilde{\bm{a}} \in \Z^{d+1}$ to be $\widetilde{\bm{a}} = (a_1,\ldots,a_d,1)$.
More generally, if $\bm{a} = (a_1,\ldots, a_d) \in \Z^d$ and $h\in \N$, we will write $(\bm{a},h)$ instead of $(a_1,\ldots, a_d,h)$, and refer to $h$ as the \emph{height} of this point.
The following notions are fundamental to our work:

\begin{defn}
Define the \emph{cone over $A$} to be
\begin{equation}
\c_A := \Span_{\N}\{\widetilde{\bm{a}}_1, \ldots, \widetilde{\bm{a}}_k \}
=  \{n_1\widetilde{\bm{a}}_1 +  \cdots + n_k \widetilde{\bm{a}}_k : n_1,\ldots, n_k \in\N \} .
\end{equation}
To the cone $\c_A$ we associate a generating series $\c_A(t) \in \Q \llbracket t \rrbracket$:
\begin{equation} \label{cardGenSeries}
\c_A(t) := \sum_{\bm{a} \in \c_A}  t^{\text{height}(\bm{a})} .
\end{equation}
\end{defn}

It may be more intuitive to think about $\c_A$ geometrically:
the points at height $h$ in $\c_A$ form a copy of $hA$, embedded into $\Z^{d+1}$.
Viewed from this perspective, we see that $\c_A(t)$ is simply the generating function of $hA$:
\begin{equation}
\c_A(t) = \sum_{h\geq 0} |hA| t^h .
\end{equation}
Our goal is to partition $\c_A$ into simple geometric pieces, and then use this to decompose $\c_A(t)$ into a sum of nice rational functions. Once this is accomplished, we'll be able to determine $|hA|$ for all values of $h$.

The following example captures the key components of our approach.
Let ${A = \{0,1,7,8\}}$; the  first few levels of $\c_A$ are illustrated in Figure \ref{fig:4elmtex}.
Note the two boundary rays are spanned by the vectors $(0,1)$ and $(8,1)$, which are linearly independent.
Therefore the lattice $\Lambda = \Span_\Z \{(0,1),(8,1)\}$ has finite index in $\Z^2$, so we can partition $\c_A$ into finitely many equivalence classes modulo $\Lambda$.

\begin{figure}
    \centering
    \includegraphics[width = 12cm]{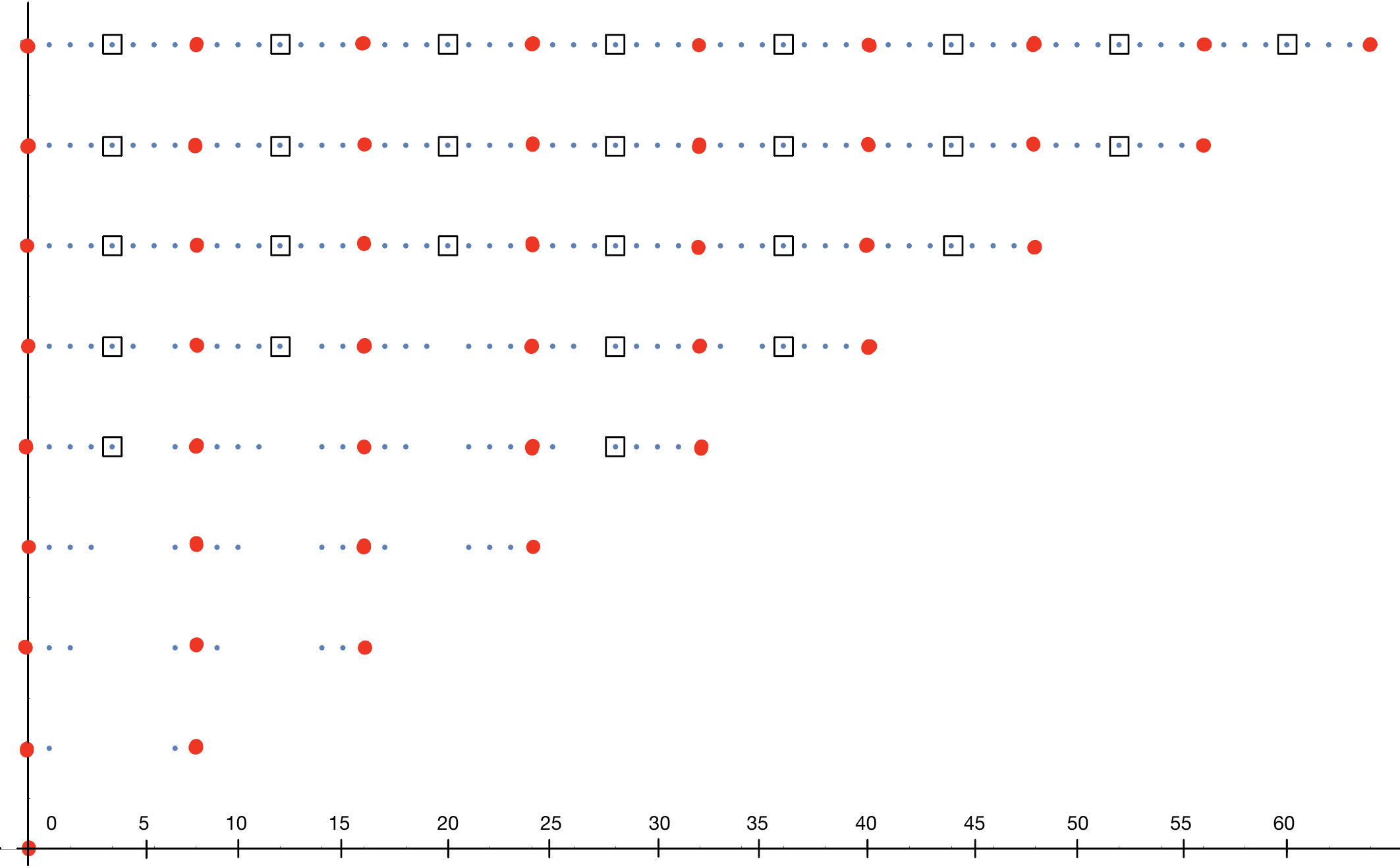}
    \caption{The cone $\c_A$ over $A = \{0,1,7,8\}$. Elements lying above the residue class 0 mod 8 are labeled with bold circles, elements lying above the residue class 4 mod 8 are labeled with hollow squares, and the other elements are simply dots.}
    \label{fig:4elmtex}
\end{figure}

Given $0 \leq m < 8$, let $\s_m$ denote the points of $\c_A$ lying in the residue class of $(m,1)$.
The equivalence class $\s_0$ is simple to understand: it is just the set $\Lambda^+ := \Span_\N \{(0,1),(8,1)\}$, represented by bold circles in Figure \ref{fig:4elmtex}.
By the geometric series formula, the generating series of $\s_0$ is simply
\[
\sum_{\bm{a} \in \s_0}  t^{\text{height}(\bm{a})} = \frac{1}{(1-t)^2}.
\]
The residue class $\s_4$ consists of the hollow squares in Figure \ref{fig:4elmtex}, and can be viewed as a union of two translates of $\Lambda^+$:
\[
\s_4 = \bigg((4,4) + \Lambda^+ \bigg) \cup \bigg((28,4) + \Lambda^+ \bigg).
\]
These two cones are not disjoint, with intersection at $(28,7) + \Lambda^+$.
Inclusion-exclusion implies
\[
\sum_{\bm{a} \in \s_4}  t^{\text{height}(\bm{a})} = \frac{t^4}{(1-t)^2} + \frac{t^4}{(1-t)^2}  - \frac{t^7}{(1-t)^2}  = \frac{2t^4 - t^7}{(1-t)^2} .
\]
Making similar calculations for the remaining residue classes of $\c_A$ and adding the corresponding generating functions together, one finds
\begin{align*}
\c_A(t) &= \frac{1 + 2t + 2t^2 + 2t^3+2t^4 + 2t^5 + 2t^6 - 5t^7}{(1-t)^2}.
\end{align*}
Expanding this as a power series, we conclude
\[
\sum_{h\geq 0 } |hA|t^h
= \c_A(t)
= -5 t - 8 t^2 - 9 t^3 - 8 t^4 - 5 t^5 + \sum_{h\geq 0 } (8h + 1)t^h .
\]
We deduce from this a totally explicit version of Khovanskii's theorem for the set $A = \{0,1,7,8\}$:
$|hA| = 8h  + 1$ for $h \geq 6$.
Our goal in the sequel will be to adapt this approach to more general sets $A$.

As a first step, consider any 3-element set $A \subset \Z$; after translating and dilating, we may assume $A = \{0,a, b\}$ where $0< a < b$ and $a$ and $b$ are relatively prime.
Since $a$ and $b$ are relatively prime, all of the elements $(m a, m)$ for $0 \leq m < b$ are distinct modulo the lattice spanned by $(b,1)$ and $(0,1)$.
Furthermore, they necessarily generate the residue class modulo $\Lambda$ they lie in:
\[
\s_{ma} = (m a, m) + \Span_\N\{(0,1), (0,b)\}.
\]
Now because the number of residue classes modulo $\Lambda$ is exactly $b$, it follows that
\begin{align*}
\c_A(t) &= \frac{1 + t + t^2 + \cdots + t^{b-1}}{(1-t)^2} = \frac{1 - t^b}{(1 - t)^3}
\end{align*}
Expanding $\c_A(t)$ as a power series gives that
\begin{align*}
\sum_{h\geq 0} |hA| t^h &= \frac{1 - t^b}{(1 - t)^3} = \sum_{h\geq 0} \binom{h + 2}{h} t^h - \sum_{h\geq 0}  \binom{h + 2}{h} t^{h+b} \\
&= \sum_{h\geq 0} \binom{h + 2}{2} t^h  - \sum_{h\geq b}  \binom{h - b + 2}{2} t^{h}.
\end{align*}
Equating coefficients, we find
\[
|hA| = \binom{h + 2}{2}
\qquad \text{whenever } 0 \leq h < b - 2
\]
and
\[
|hA| = \binom{h + 2}{2} - \binom{h- b + 2}{2}
\qquad \text{whenever } h \geq b - 2.
\]
These formulas generalize to arbitrary dimension, as was stated in Theorem \ref{thm:dplus2}.
We conclude this section by proving Theorem \ref{thm:dplus2} in the special case that $\Delta_A$ is a simplex.

\begin{proof}[Proof of Theorem \ref{thm:dplus2} for simplicial sumsets]
Denote the vertices of $\Delta_A$ by $\bm{v}_1,\ldots, \bm{v}_{d+1}$, and without loss of generality suppose the $(d+2)^{\text{nd}}$ element of $A$ is $\bm
{0}$.
Set $\Lambda := \Span_\Z \{\widetilde{\bm{v}}_1,\ldots, \widetilde{\bm{v}}_{d+1}\}$ and $\Lambda^+ := \Span_\N \{\widetilde{\bm{v}}_1,\ldots, \widetilde{\bm{v}}_{d+1}\}$.
It is a well-known result in the geometry of numbers that $\Z^{d+1}/\Lambda$ can be identified with the set of lattice points in the fundamental domain of $\Lambda$, and that the number of lattice points lying in the fundamental domain of $\Lambda$ is the determinant of the matrix whose columns are $\widetilde{\bm{v}_i}$ \cite[Ch. 6, Sec. 1]{AddNT}. Thus,
\[
|\Z^{d+1}/\Lambda| = \vol(\Delta_A)\cdot d!.
\]
Because $A$ generates $\Z^d$ it follows that all the vectors $(\bm{0}, m)$ with $0\leq m < \vol(\Delta_A)\cdot d!$ are distinct modulo $\Lambda$, whence
\[
\c_A = \bigsqcup_{m = 0}^{\vol(\Delta_A) \cdot d! - 1}\bigg( (\bm{0}, m) + \Lambda^+ \bigg).
\]
This implies
\[
\c_A(t) = \frac{1 + t + \cdots + t^{\vol(\Delta_A) \cdot d! - 1}}{(1-t)^{d+1}} = \frac{1 - t^{\vol(\Delta_A) \cdot d!}}{(1-t)^{d+2}}.
\]
Now observe that
\[
\frac{1}{(1 - t)^{d + 2}} =  \sum_{h\geq 0} \binom{h+d+1}{h} t^h = \sum_{h\geq 0} \binom{h+d+1}{d+1} t^h
\]
while
\begin{align*}
\frac{t^{\vol(\Delta_A) \cdot d!}}{(1 - t)^{d + 2}} &=  \sum_{h\geq 0} \binom{h+d+1}{d+1} t^{h+\vol(\Delta_A) \cdot d!} =  \sum_{h\geq \vol(\Delta_A) \cdot d!} \binom{h-\vol(\Delta_A) \cdot d!+d+1}{d+1} t^{h} .
\end{align*}
The claim follows.
\end{proof}

\section{Effective Khovanskii for simplicial sumsets: Proof of Theorem \ref{thm:SimplexSize}}\label{sec:GenSeries}

In the last section we proved a completely explicit version of Khovanskii's theorem over $\Z^d$ in the special case that $A$ consists of $d+2$ points and the convex hull of $A$ is a simplex.
In this section we drop the condition on the size of $A$ and try to push our methods further. This comes at a cost---the geometry of the cone $\c_A$ becomes more complicated---but we will still be able to obtain an effective bound on the phase transition (i.e.\ what `sufficiently large' means) in Khovanskii's theorem.

Let $A\subset \Z^d$ be a finite set such that $A - A$ generates $\Z^d$ additively and $\Delta_A$ is a simplex.
Denote the $d + 1$ vertices of $\Delta_A$ by $\bm{v}_1 ,\ldots, \bm{v}_{d+1}$. These span a lattice
\[
\Lambda := \Span_\Z\{\widetilde{\bm{v}}_1,\ldots, \widetilde{\bm{v}}_{d+1}\} \subset \Z^{d+1}
\]
of finite index in $\Z^{d+1}$.
(Recall that $\widetilde{\bm{v}}$ denotes the lift of $\bm{v}$ to height 1 in $\Z^{d+1}$.)
We will also be interested in the subset
\[
\Lambda^+ := \Span_\N \{\widetilde{\bm{v}}_1,\ldots, \widetilde{\bm{v}}_{d+1}\} \subset \Lambda.
\]
Finally, we denote by $\Pi$ the set of integer lattice points lying in the fundamental domain of $\Lambda$; in symbols,
\[
\Pi := \left\{\sum_{i= 1}^{d+1} \l_i \widetilde{\bm{v}}_i : 0\leq \l_i < 1\right\} \cap \Z^{d+1} .
\]

We now partition $\c_A$ according to the residue classes (mod $\Lambda$), each of which can be represented by an element of $\Pi.$
Given $\bm{\pi} \in \Pi$, define $\s_{\bm{\pi}}$ to be the set of elements of $\c_A$ that are congruent to $\bm{\pi}$ modulo $\Lambda$.
We call $(\bm{g},N) \in \s_{\bm{\pi}}$ a \emph{minimal element} if $(\bm{g},N) - \widetilde{\bm{v}}_i$ does not lie in $\c_A$ for any $i$.

\begin{rmk}
The set $\s_{\bm{\pi}}$ can be given the structure of a partially ordered set, where $(\bm{a}, N) \leq (\bm{b},M)$ if and only if  $(\bm{b},M) - (\bm{a},N) \in \Lambda^+$.
Our definition of minimal element coincides with the minimal elements of $\s_{\bm{\pi}}$ as a poset.
\end{rmk}

As in the previous section, we associate to each residue class $\bm{\pi} \in \Pi$ a generating series
\[
\s_{\bm{\pi}}(t) = \sum_{\bm{a} \in \s_{\bm{\pi}}}  t^{\text{height}(\bm{a})}.
\]
In the examples from the previous section, $\s_{\bm{\pi}}(t)$ was a rational function of the form ${P(t)/(1-t)^d}$, and we will soon see (Lemma \ref{lem:genseriesdegbound}) that this is always the case.
In order to obtain an effective version of Khovanskii's theorem, it will be necessary to obtain bounds on the degree of $P$. We do this in two steps: first, we control the heights of the minimal elements, and then we relate the degree of $P$ to the minimal elements of $\s_\pi$.

\begin{lemma}\label{lem:minheightbound}
If $(\bm{\a},M)$ is a minimal element of $\s_{\bm{\pi}}$, then
\[
M  \leq \vol(\Delta_A)\cdot d! - 1.
\]
In particular there are finitely many minimal elements.
\end{lemma}
\begin{proof}
Without loss of generality assume that $\bm{0}$ is a vertex of $\Delta_A$, say $\bm{v}_{d+1} = 0$.
By assumption we may write $\bm{\a} = \bm{a}_1 + \cdots + \bm{a}_M$ with each $\bm{a}_i \in A$.
We claim that the $M$ subsums
\[
\bm{a}_1, \; \bm{a}_1 + \bm{a}_2, \; \bm{a}_1 + \bm{a}_2 + \bm{a}_3, \; \ldots, \; \bm{a}_1 + \cdots + \bm{a}_M
\]
are all distinct modulo $\Lambda$; since the number of nonzero residue classes modulo $\Lambda$ is $\vol(\Delta_A)\cdot d! - 1$, the claim follows.

Suppose instead that $\bm{a}_1 + \cdots + \bm{a}_m$ and $\bm{a}_1 + \cdots + \bm{a}_n$ were congruent modulo $\Lambda$ for some $m < n$. Then $\bm{a}_{m+1} + \cdots + \bm{a}_n \in \Lambda$.
Since each $\bm{a}_i$ lies in $\Delta_{\Lambda^+}$ and $\Delta_{\Lambda^+}$ is convex, we must have $\bm{a}_{m+1} + \cdots + \bm{a}_n \in \Delta_{ \Lambda^+} \cap \Lambda = \Lambda^+$.
It follows that there exist $k_i \in \N$ such that
\[
\bm{a}_{m+1} + \cdots + \bm{a}_n = \sum_{i = 1}^{d} k_i \bm{v}_i .
\]
Writing each $\bm{a}_j$ in barycentric coordinates ${\bm{a}_j = \sum_{i = 1}^d \l_{i,j} \bm{v}_i}$ with $\l_{i,j} \geq 0$ and $\sum_{i=1}^d \lambda_{i,j} \leq 1$, we see that
\[
\bm{a}_{m+1} + \cdots + \bm{a}_n = \sum_{i = 1}^d  \left(\sum_{j = m+1}^n \l_{i,j}\right) \bm{v}_i  = \sum_{i = 1}^{d} k_i \bm{v}_i.
\]
Since the nonzero vertices of $\Delta_A$ are linearly independent we deduce
\[
\sum_{i = 1}^d k_i =   \sum_{j = m+1}^n   \sum_{i = 1}^d \l_{i,j} \leq \sum_{j = m+1}^n 1 = n - m .
\]
But this contradicts the minimality of $(\bm{\alpha},M)$! To see this, set $\bm{\b} := \bm{\alpha} - (\bm{a}_{m+1} + \cdots + \bm{a}_n)$ and note that
\[
(\bm{\alpha},M) - (\bm{\beta},M - (n-m)) = \left(\sum_{i = 1}^d k_i \bm{v}_i,  n - m\right) \in \Lambda^+
\]
since $\sum_i k_i \leq n - m$ and $\bm{0}$ is a vertex of $\Delta_A$. This implies $(\bm{\a},M)$ is not minimal.
\end{proof}

\begin{lemma}\label{lem:genseriesdegbound}
Suppose the minimal elements of $\s_{\bm{\pi}}$ are $(\bm{g}_1,H_1), \ldots, (\bm{g}_n, H_n)$.
Then we can write
\[
\s_{\bm{\pi}}(t) = \frac{P(t)}{(1-t)^{d+1}}
\]
for some $P \in \Q[t]$ with
\(
\deg P \leq (d+1) \cdot \max_{i} (H_i) - d .
\)
\end{lemma}

\begin{rmk}
When $n = 1$, we simply have $P(t) = t^{H_1}$.
\end{rmk}

\begin{proof}
As before assume $\bm{v}_{d+1} = \bm{0}$.
Furthermore, we may assume that the elements $\bm{g}_i$ are not congruent to $\bm{0}$ (mod $\Lambda$) since the origin in $\Z^{d+1}$ is the unique minimal element of $\s_{\bm{0}}$.
By assumption we may write
\[
\s_{\bm{\pi}} = \bigcup_{i = 1}^n \bigg( (\bm{g}_i,H_i) + \Lambda^+\bigg) .
\]
Inclusion-exclusion implies that $\s_{\bm{\pi}}(t)$ is a weighted sum of the generating series of all possible intersections of the sets $(\bm{g}_i,H_i) + \Lambda^+$.
Now observe that for each $I \subseteq \{1,\ldots, n\}$, we can write
\[
\bigcap_{i \in I} \bigg( (\bm{g}_i,H_i) + \Lambda^+\bigg) = (\bm{g}_I, H_I) + \Lambda^+
\]
for some $\bm{g}_I\in \c_A$ and $H_I\in \N$.
Since the generating series of $(\bm{g}_I, H_I) + \Lambda^+$ is simply
\(\displaystyle
\frac{t^{H_I}}{(1-t)^{d+1}},
\)
it suffices to bound $H_I$ as $I$ varies over all subsets of $\{1,\ldots,n\}$.
In fact, we only need to bound $H_I$ with $I = \{1,\ldots, n\}$, since
$(\bm{g}, H) + \Lambda^+ \subset (\bm{g}', H') + \Lambda^+$ implies $H \geq H'$.

Without loss of generality assume that $\max_i H_i = H_1$.
Since $\bm{v}_{d+1} = \bm{0}$, for each $i > 1$ there exist integers $m_{i,1},\ldots, m_{i,d}$ such that
\[
\bm{g}_i - \bm{g}_1 = \sum_{j=1}^{d} m_{i,j} \bm{v}_j .
\]
Now for each $j,$ let $m_j = \max_{i}(|m_{i,j}|)$.
We claim that $m_j \leq H_1 - 1$ for each $j$.
To this end, if we denote by $\overline{\bm{\pi}}$ the projection of $\bm{\pi}$ to $\Z^{d}$ then we may write
\[
\bm{g}_1 = \bar{\bm{\pi}} + \sum_{j = 1}^d n_{1,j} \bm{v}_j, \qquad \bm{g}_i = \bar{\bm{\pi}} + \sum_{j = 1}^d n_{i,j} \bm{v}_j
\]
for integers  $n_{1,j} , n_{i,j}$.
Next observe that $\bm{g}_1$ and $\bm{g}_i$ lie in the convex hull of $H_1 A$, so it follows that $n_{1,j}$ and $n_{i,j}$ are nonnegative and that
\[
\sum_{j = 1}^d n_{1,j}  \leq H_1 , \qquad \sum_{j = 1}^d n_{i,j} \leq H_1.
\]
In fact both of these inequalities are strict since $\bar{\bm{\pi}} \neq \bm{0}$, so the they hold with $H_1 - 1$ in place of $H_1$.
In particular it follows that $0 \leq n_{i,j}, n_{1,j} \leq H_1 - 1$.
Therefore $|m_{i,j}| = |n_{i,j} - n_{1,j}| \leq H_1 - 1$, hence the claim since $i$ was arbitrary.

Now let
\[
\bm{\a} := \bm{g}_1 + \sum_{i=1}^d m_j \bm{v}_j
\]
and observe that $\bm{\a} \in \Big((d + 1)H_1 - d\Big) A$ since $\bm{g}_1 \in H_1 A$ and each $m_j \leq H_1 - 1$.
We claim in fact that
\[
(\bm{\a}, (d + 1)H_1 - d) \in \bigcap_{i = 1}^n \bigg( (\bm{g}_i,H_i) + \Lambda^+\bigg).
\]
It suffices to show that $\bm{\a} - \bm{g}_i \in \Lambda^+$ for each $i$.
Clearly $\bm{\a} - \bm{g}_1 \in \Lambda^+$, and otherwise
\[
\bm{\a} - \bm{g}_i = \sum_{j=1}^{d} (m_j - m_{i,j}) \bm{v}_j \in \Lambda^+
\]
because $m_j \geq m_{i,j}$ for each $i$.
Therefore
\[
(\bm{\a},(d + 1)H_1 - d) + \Lambda^+ \subseteq \bigcap_{i = 1}^n \bigg( (\bm{g}_i,H_i) + \Lambda^+\bigg) = (\bm{g}_I,H_I) + \Lambda^+,
\]
so $H_I \leq (d + 1)H_1 - d$.
\end{proof}

\noindent
With these results in hand, it's not too difficult to establish an effective version of Khovanskii's theorem for the case that the convex hull of $A$ is a simplex:

\begin{proof}[Proof of Theorem \ref{thm:SimplexSize}]
By the previous lemmas, we can write
\[
\c_A(t) = \frac{P(t)}{(1-t)^{d+1}}
\]
where $\deg P \leq  \vol(\Delta_A)\cdot (d+1)! - 1 - 2d$.
The division algorithm furnishes $R,Q\in \Q[t]$ with $\deg R \leq d$ and $\deg Q = \deg P - d - 1$ such that
\[
\c_A(t) =
\frac{P(t)}{(1-t)^{d+1}} =
Q(t) + \frac{R(t)}{(1-t)^{d+1}}.
\]
Write $R(t) = a_0 + a_1 t + \cdots + a_d t^d$ where $a_i$ are (possibly zero) rational numbers,
and observe that
\begin{align*}
    \frac{R(t)}{(1-t)^{d+1}} &= (a_0 + a_1 t + \cdots + a_d t^d) \sum_{n \geq 0} \binom{n + d}{d} t^n \\
    &= a_0 \sum_{h \geq 0} \binom{h + d}{d} t^h + a_1 \sum_{h \geq 1} \binom{h - 1 + d}{d} t^h + \cdots + a_d\sum_{h\geq d} \binom{h}{d}t^h \\
    &= \sum_{h \geq 0} \left(\sum_{k=0}^d a_k \binom{h+d-k}{d}\right) t^h.
\end{align*}
The final equality holds since $\binom{h+d-k}{d}$ vanishes for $0 \leq h \leq k$.
In particular, it follows that there is some $p\in \Q[x]$ such that
\[
\frac{R(t)}{(1-t)^{d+1}} = \sum_{h\geq 0} p(h) t^h .
\]
This agrees with $\c_A(t)$ for all terms beyond $t^{\deg Q}$, and the claim follows.
\end{proof}

\section{A local-global structure theorem for sumsets: Proof of Theorem \ref{thm:SimplexStructure}}\label{sec:structureThm}

In the previous section we proved results about the structure of the cone $\c_A$ and deduced information about the cardinality of $hA$ for all sufficiently large $h$.
The goal of this section is to deduce information about the structure of $hA$ instead.
For example, one consequence of our work will be an explicit description of $hA$ for {all} $h \in \N$ in terms of the minimal elements of the cone $\c_A$:

\begin{prop} \label{prop:ExplicitSumsets}
Suppose $A \subset \Z^d$ has convex hull $\Delta_A$ a simplex.
Say the vertices of $\Delta_A$ are $\bm{v}_1, \bm{v}_2, \ldots, \bm{v}_{d+1}$, and denote the minimal elements of the cone $\c_A$ by $(\bm{g}_1,H_1), (\bm{g}_2,H_2), \ldots$
Then for all $h \in \N$,
\[
hA =
\bigcup_j
\bigg\{
    \bm{g}_j + \sum_{i \leq d+1} k_i \bm{v}_i :
        k_i \in \N \text{ for all $i$ and }
        \sum_{i \leq d+1} k_i = h - H_j
\bigg\} .
\]
\end{prop}
\begin{rmk}
We are slightly abusing our terminology, since we previously defined \emph{minimal element} only for a given residue class $\s_{\bm{\pi}}$. The collection of all the minimal elements from all the $\s_{\bm{\pi}}$ is what we mean by the minimal elements of $\c_A$.
\end{rmk}

While this proposition completely describes all iterated sumsets of $A$, it does so in terms of the minimal elements of $\c_A$, whose structure remains elusive. (Computationally the minimal elements can be determined without much difficulty in view of Lemma \ref{lem:minheightbound}.)
Nonetheless, the fact that we are able to prove such a result for all $h \in \N$ will prove critical in our proof of Theorem \ref{thm:SimplexStructure}, the main goal of this section.

Our first step is to rephrase Theorem \ref{thm:SimplexStructure} in a geometric form. To this end, we introduce a new tool to our kit:

\begin{defn}
Given a vertex $\bm{v}$ of the convex hull of $A$, define the \emph{tangent cone at $\bm{v}$} by
\begin{equation}
    T_{\bm{v} } (A) := \bigcup_{h\geq 0} h(A - \bm{v}).
\end{equation}
\end{defn}

\noindent
Thus, for example, $T_{\bm{0}}(A) = \bigcup_{h \geq 0} hA$, the projection of the cone $\c_A$ onto $\Z^d$ that deletes the final coordinate.

Starting with the identity $hA = h \bm{v} + h(A-\bm{v})$, notice that in order for $\bm{a}\in\Z^d$ to lie in $hA$ for some $h$, it must lie in $h\bm{v} + T_{\bm{v}}(A)$ for each vertex.
Thus, in a sense, the tangent cones take into account local obstructions near  each vertex of $\Delta_{hA}$ to writing an element of $\Z^d$ as a positive linear combination of elements of $A$.
For the rest of this section, we will denote the  vertices of $\Delta_A$ by $\bm{v}_1,\ldots, \bm{v}_{d+1}$; without loss of generality, $\bm{v}_{d+1} = \bm{0}$. It is immediate that
\[
    hA \subseteq \bigcap_{i = 1}^{d+1} \Big(h\bm{v}_i + T_{\bm{v}_i } (A) \Big) .
\]
The content of Theorem \ref{thm:SimplexStructure} is that the reverse inclusion holds for large $h$.
In other words, for large $h$, the global structure of $hA$ is completely determined by the local structure at each vertex of $\Delta_{hA}$.

Our approach will follow that of the previous section, except that we will consider not just the single cone $\c_A$ but rather the $d+1$ different cones $\c_i := \c_{A-\bm{v}_i}$.
To each of these cones we can associate quantities analogous to those in section \ref{sec:GenSeries}: let $\Lambda_i$ denote the lattice in $\Z^{d+1}$ spanned by $(\bm{v}_j - \bm{v}_i,1)$ with $1\leq j \leq d+1$, denote by $\Lambda_i^+$ the set of nonnegative integer linear combinations of $(\bm{v}_j - \bm{v}_i,1)$ with $1\leq j \leq d+1$, and let $\Pi_i$ denote the set of lattice points in the fundamental domain of $\Lambda_i$. For any given $i$ and each $\bm{\pi} \in \Pi_i$, let $\s_{\bm{\pi},i}$ denote the elements of $\c_i$ that are congruent to $\bm{\pi}$ modulo $\Lambda_i$.

This notation allows us to describe the tangent cones $T_{\bm{v}_i}(A)$ in terms of the minimal elements of $\c_i$.
Let $m(\bm{\pi},i)$ be the total number of minimal elements of $\s_{\bm{\pi},i}$, and enumerate these minimal elements in the form $\left(\bm{g}_{\bm{\pi},i}^{1},H_{\bm{\pi},i}^{1}\right), \left(\bm{g}_{\bm{\pi},i}^{2},H_{\bm{\pi},i}^{2}\right), \ldots$
In particular,
\[
\s_{\bm{\pi},i} = \bigcup_{j = 1}^{m(\bm{\pi},i)} \bigg( \left(\bm{g}_{\bm{\pi},i}^{j},H_{\bm{\pi},i}^{j}\right) + \Lambda_i^+ \bigg) ,
\]
whence
\begin{equation}\label{eq:TCDoubleUnion}
T_{\bm{v}_i}(A) = \bigsqcup_{\bm{\pi} \in \Pi_i} \bigcup_{j = 1}^{m(\bm{\pi},i)} \left\{ \bm{g}_{\bm{\pi},i}^{j} + \sum_{k = 1}^{d+1} n_k (\bm{v}_k  - \bm{v}_i) : n_k \in \N \right\} .
\end{equation}

One of the difficulties in working with tangent cones is that distinct sets may have the same tangent cone.
For example, if $A = \{0,1,3,4\}$ and $B = \{0,1,2,3,4\}$ then $T_0(A) = T_0(B) = \N$ and $T_b(A) = T_b(B) = -\N$, so the tangent cones lose some information about the underlying set. In particular, the tangent cones only determine the long term behavior of $hA$.
In order to obtain the explicit bound on the phase transition in Theorem \ref{thm:SimplexStructure}, it turns out we will need insight into the structure of $hA$ for \emph{all} $h \in \N$.
In the example of $A = \{0,1,3,4\}$, even though the elements $(2,2)$ and $(6,2)$ are both minimal elements of residue class $(2,2)$ in $\c_A$, all of the elements of $T_0(A)$ equivalent to 2 mod 4 can be expressed in the form $2 + 4n$ so we can think of 2 as a minimal element of the points in $T_0(A)$ congruent to 2 mod 4.
In other words, the tangent cone at 0 fails to recognize 6 as a minimal element mod 4 (see Figure \ref{fig:Proj}).
In one dimension, the natural ordering on $\Z$ lets one get away with only knowing the smallest minimal elements.
For higher dimensions, however, one must keep track of all of the minimal elements, which the cones $\c_i$ allow us to do;
this is what permits us to make the structure theorem given in \cite{GranvilleShakan} effective for dimensions greater than 1.

\begin{figure}
    \centering
    \includegraphics[width = 12cm]{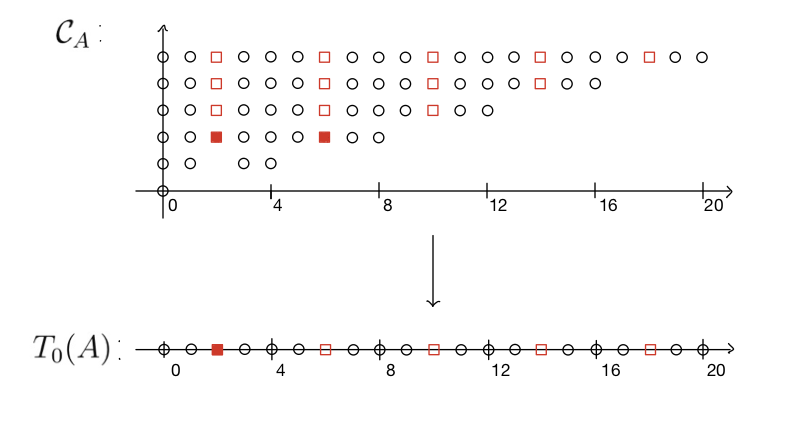}
    \caption{The cone $\c_A$ over $A = \{0,1,3,4\}$ lying above the tangent cone $T_0(A)$ at 0. The points lying above the residue class 2 mod 4 are labeled with boxes, and the minimal elements are labeled with shaded boxes. In $\c_A$, it is clear that there are two minimal elements, but we lose this distinction upon projecting onto $T_0(A)$.}
    \label{fig:Proj}
\end{figure}

It turns out that for any choices of index $i,j$, the minimal elements of $\c_i$ and $\c_j$ are closely related to one another. To state this as transparently and concretely as possible, we adopt our notation from section \ref{sec:GenSeries}: let $\Lambda$ be the lattice in $\Z^{d+1}$ generated by $\widetilde{\bm{0}}, \widetilde{\bm{v}}_1, \widetilde{\bm{v}}_2, \ldots, \widetilde{\bm{v}}_d$, fix a lattice point $\bm{\pi}$ in the fundamental domain of $\Lambda$, and let $\s_{\bm{\pi}}$ be the collection of all points of $\c_A$ equivalent to $\bm{\pi} \mod{\Lambda}$. Denote the minimal elements of $\s_{\bm{\pi}}$ by $(\bm{g}_1, H_1), (\bm{g}_2, H_2), \ldots, (\bm{g}_n, H_n)$.

\begin{lemma}\label{lem:minElmtSym}
For any $1\leq i \leq d$, the number of minimal elements in $\s_{\bm{\pi}-\widetilde{\bm{v}}_i , i}$ is precisely $n$,
and (after suitably permuting the order of the minimal elements) we have
\[
\bm{g}_{\bm{\pi} - \widetilde{\bm{v}}_i,i}^{j} = \bm{g}_j - H_j \bm{v}_i
\qquad \text{and} \qquad
H_{\bm{\pi} - \widetilde{\bm{v}}_i ,i}^{j} = H_j
\]
for all $j \leq n$.
\end{lemma}

\begin{proof}
Observe that $h(A - \bm{v}_i) = hA - h \bm{v}_i$ furnishes a bijection between $hA$ and $h(A - \bm{v}_i)$.
Thus if $\left(\bm{g}_{j}, H_{j}\right)$ is a minimal element of $\s_{\bm{\pi}}$, then $\left(\bm{g}_{j} -  H_{j} \bm{v}_i , H_{j}\right)$ must be a minimal element of $\c_i$ congruent to $\bm{\pi} - (\bm{v}_i,0)$ modulo $\Lambda_i$ and vice versa.
The claim now follows since $\widetilde{\bm{0}} \in \Lambda_i$ for each $1\leq i \leq d+1$, so $\bm{\pi} - (\bm{v}_i,0)$ is equivalent to $\bm{\pi} - \widetilde{\bm{v}}_i$ modulo $\Lambda_i$.
\end{proof}

\noindent With this in hand, we can now prove our structure theorem for $hA$:

\begin{proof}[Proof of Theorem \ref{thm:SimplexStructure}]
Our goal is to show that for all $h \geq \vol(\Delta_A) \cdot (d+1)!  - 2 - 2d$,
\begin{equation}\label{eq:RCRevInc}
\bigcap_{i = 1}^{d+1} \Big(h\bm{v}_i +  T_{\bm{v}_i}(A) \Big)
\subseteq hA.
\end{equation}
Let $\Gamma = \Span_\Z \{\bm{v}_1,\ldots, \bm{v}_d\} \subseteq \Z^d$ and $\Gamma^+ = \Span_\N \{\bm{v}_1,\ldots, \bm{v}_d\} \subset \Z^d$.
Fix any lattice point $\bm{\pi}$ in the fundamental domain of $\Lambda$, and consider the set $\s_{\bm{\pi}}$ consisting of all points of $\c_A$ that are equivalent to $\bm{\pi}$ modulo $\Lambda$.
Denote the minimal elements of $\s_{\bm{\pi}}$ by
$(\bm{g}_1,H_1), (\bm{g}_2,H_2), \ldots, (\bm{g}_n,H_n)$.
Note that the assumption that $\bm{0}$ is a vertex of $\Delta_A$ implies that $\widetilde{\bm{0}} \in \Lambda$, so every residue class of $\Z^{d+1}$ modulo $\Lambda$ contains a representative in $\Gamma$. In particular, $\bm{g}_i \equiv \bm{g}_j \mod \Gamma$ for any $i,j$.

Having set the notation, we turn to the proof. Let
\[
L_{\bm{\pi}}(h) :=
\left(
\bigcap_{i = 1}^{d+1} \Big(h\bm{v}_i +  T_{\bm{v}_i}(A) \Big)
\right)
\cap
(\bm{g}_1 + \Gamma)
\]
and
\[
R_{\bm{\pi}}(h) :=
hA \cap (\bm{g}_1 + \Gamma) .
\]
Informally, $L_{\bm{\pi}}(h)$ is the ``$\bm{\pi}$-part'' of the left hand side of \eqref{eq:RCRevInc}, and $R_{\bm{\pi}}(h)$ is the ``$\bm{\pi}$-part'' of the right hand side.
Since $\bm{\pi}$ was arbitrarily chosen, to prove \eqref{eq:RCRevInc}
it suffices to prove that $L_{\bm{\pi}}(h) \subseteq R_{\bm{\pi}}(h)$ for all $h \geq \vol(\Delta_A) \cdot (d+1)!  - 2 - 2d$.

We rewrite these two quantities, beginning with $L_{\bm{\pi}}(h)$.
Combining (\ref{eq:TCDoubleUnion}) with Lemma \ref{lem:minElmtSym},
we deduce
\begin{equation} \label{eq:Thpi}
\begin{split}
L_{\bm{\pi}}(h)
&= \bigcap_{i \leq d+1} \bigcup_{j = 1}^{m(\bm{\pi}-\widetilde{\bm{v}}_i,i)} \left\{ h \bm{v}_i + \bm{g}_{\bm{\pi} - \widetilde{\bm{v}}_i,i}^{j} + \sum_{k = 1}^{d+1} n_{i,k} (\bm{v}_k  - \bm{v}_i) : n_{i,k} \in \N \right\} \\
&=
\bigcap_{i \leq d + 1} \bigcup_{j \leq n} \left\{ h \bm{v}_i +  \bm{g}_{j} - H_j \bm{v}_i  + \sum_{k = 1}^{d+1} n_{i,k} (\bm{v}_k - \bm{v}_i) : n_{i,k} \in \N\right\}
\end{split}
\end{equation}
Next we turn to $R_{\bm{\pi}}(h)$.
Note that any point of $R_{\bm{\pi}}(h)$ has the form
$\bm{g}_j + \sum_{i = 1}^{d} k_i \bm{v}_i$, which lives in $hA$ whenever
$H_j + \sum_{i = 1}^{d} k_i  \leq h$.
Since $(\bm{g}_j , H_j)$ is a minimal element, we deduce
\begin{equation}\label{eq:hApi}
R_{\bm{\pi}}(h) = \bigcup_{j \leq n}
\left\{\bm{g}_j + \sum_{i = 1}^{d} k_i \bm{v}_i : k_i \in \N \text{ and } \sum_{i = 1}^{d} k_i  \leq h - H_j \right\}.
\end{equation}
Our strategy from here will be to dissect $L_{\bm{\pi}}(h)$ into two pieces, one that only depends on $\bm{\pi}$ and lives in $R_{\bm{\pi}}(h)$ for sufficiently large $h$, the other depending on $h$ in a tame enough way that it lives in $R_{\bm{\pi}}(h)$ for all $h \in \N$.

Exactly as in the proof of Lemma \ref{lem:genseriesdegbound}, we may write
\begin{equation}\label{eq:intersectMinmlCones}
\bigcap_{j \leq n} \bigg((\bm{g}_j, H_j) + \Lambda^+ \bigg) = (\bm{g}_{\bm{\pi}}, H_{\bm{\pi}}) + \Lambda^+
\end{equation}
with $H_{\bm{\pi}} \leq \vol(\Delta_A)\cdot(d+1)! - 1 - 2d$.
Set
\[
\p_{\bm{\pi}} := \left\{\bm{g}_{\bm{\pi}} - \sum_{i = 1}^{d} n_i \bm{v}_i  : n_i\in \Z_{>0}  \right\} \cap L_{\bm{\pi}}(h).
\]
It immediately follows that $\p_{\bm{\pi}} \subseteq R_{\bm{\pi}}(h)$ whenever $h \geq H_{\bm{\pi}} - 1$.
We claim that
${L_{\bm{\pi}}(h) \setminus \p_{\bm{\pi}} \subseteq R_{\bm{\pi}}(h)}$
for all $h \in \N$, thus completing the proof.

Pick $\bm{a} \in L_{\bm{\pi}}(h) \setminus \p_{\bm{\pi}}$.
Since $\bm{a} \not\in \p_{\bm{\pi}}$, we may write
\[
\bm{a} = \bm{g}_{\bm{\pi}} + \sum_{k = 1}^{d} m_k \bm{v}_k
\]
where $m_k$ are integers with at least one non-negative, say $m_1 \geq 0$.
On the other hand, since $\bm{a} \in L_{\bm{\pi}}(h)$, the identity (\ref{eq:Thpi}) implies the existence of $j$ such that
\begin{equation}\label{eq:ThRepnofa}
\bm{a} = h \bm{v}_1 +  \bm{g}_{j} - H_j \bm{v}_1  + \sum_{k = 2}^{d+1} n_{k} (\bm{v}_k - \bm{v}_1) .
\end{equation}
Comparing these two expressions for $\bm{a}$, we deduce
\[
\bm{g}_{\bm{\pi}} - \bm{g}_{j} =
\left(h - H_j - m_1 - \sum_{k = 2}^{d+1} n_{k} \right)  \bm{v}_1  + \sum_{k = 2}^{d} (n_k - m_k) \bm{v}_k.
\]
But from (\ref{eq:intersectMinmlCones}) we know
$\bm{g}_{\bm{\pi}} - \bm{g}_j \in \Gamma^+$, whence
\[
h - H_j - m_1 - \sum_{k = 2}^{d+1} n_{k} \geq 0.
\]
Since $m_1 \geq 0$, it follows that
\[
\sum_{k = 2}^{d+1} n_{k} \leq h - H_j.
\]
Keeping this inequality in mind and regrouping the terms in (\ref{eq:ThRepnofa}), we conclude from (\ref{eq:hApi}) that $\bm{a}$ satisfies the membership requirements of $R_{\bm{\pi}}(h)$. This concludes the proof.
\end{proof}

\begin{rmk}
Proposition \ref{prop:ExplicitSumsets} isn't a corollary of Theorem \ref{thm:SimplexStructure}; its conclusion is stronger (holding for all $h \in \N$), and its hypotheses more relaxed (there's no assumption about $A-A$ generating $\Z^d$ additively).
It is, however, a porism: after shifting $A$ by one of the vertices in its convex hull we may assume that $\bm{0}$ is a vertex of $\Delta_A$, and the proposition follows from \eqref{eq:hApi} by taking the union over all lattice points $\bm{\pi}$ in the fundamental domain of $\Lambda$.
\end{rmk}

\section{A Brion-type formula for sumsets}
\label{sec:BrionFormula}

Recall that in section \ref{sec:GenSeries} we proved results on the cardinality of $hA$, essentially by realizing the generating function of $|hA|$ in two different ways and comparing the coefficients.
In section \ref{sec:structureThm} we explored the structure of $hA$ by other means, exploiting the relationship among the tangent cones of $A$.
The goal of this section is to demonstrate a hybrid of these approaches: to explore the structure of $hA$ by associating a generating function to the tangent cones of $A$.
The outcome will be a compact formula for computing the elements of $hA$ for all large $h$.
For simplicity we shall restrict ourselves to the dimension 1 case, but with more effort we expect our approach should generalize to arbitrary dimension. We give an indication of how to do so in section \ref{sec:FurtherDirections}, and invite the motivated reader to carry this out.

For the rest of this section we assume that $0$ is the smallest element of $A \subset \Z$ and that $\gcd A = 1$, and denote the largest element of $A$ by $b$.
The cardinality $|hA|$ can be viewed as assigning to each point of $hA$ a weight of $1$ and summing all the weights, and we can obtain more refined information about the structure of the sumset $hA$ by assigning different weights to its elements.
Introducing a formal variable $x$, we assign to any set $S \subseteq \Z$ the generating function
\begin{equation}\label{eq:sumsetXGF}
    \sigma_{S}(x) = \sum_{a \in S} x^a .
\end{equation}
For example, if $A = \{0,3,4,7\}$, then $\sigma_A(x) = 1 + x^3 + x^4 + x^7$.

Recall that the tangent cone of $A$ at $v$ is defined by
\[
    T_{v} (A) = \bigcup_{h\geq 0} h(A - v).
\]
To each tangent cone $T_{v}(A)$ we may associate the generating function $\sigma_{T_{v}(A)}(x)$, but for brevity we abuse notation and simply write
\begin{equation}\label{eq:TCXGF}
\sigma_{v}(x) = \sum_{a\in T_{v}(A)} x^{a}.
\end{equation}
We shall prove that the structure of $hA$ can be simply and compactly described in terms of $\sigma_0(x)$ and $\sigma_b(x)$. More precisely:

\begin{theorem}\label{thm:SumsetGFLocalRepThm}
Given $A \subset \Z$ with $\min{A} = 0$, $\max{A} = b$, and $\gcd A = 1$. Define $\sigma_v(x)$ as in \eqref{eq:TCXGF}.
Then both $\sigma_{0}(x)$ and $\sigma_{b}(x)$ are rational functions in $x$, and for all non-negative $h \geq 2b - 4$ we have
\begin{equation}\label{eq:sumsetBrion}
\sigma_{hA}(x) = \sigma_0(x) + x^{hb} \sigma_b(x).
\end{equation}
\end{theorem}

\begin{rmk}
This is analogous to a formula discovered by Brion \cite{Brion} that relates the lattice generating function of a convex polytope to the lattice generating functions of its tangent cones.
See section \ref{sec:FurtherDirections} for a generalization of our formula to higher dimensions.
\end{rmk}

Before presenting the proof of Theorem \ref{thm:SumsetGFLocalRepThm} we build intuition by applying it to the simple example $A = \{0,2,5\}$
mentioned in the introduction.
Observe that
\[
T_0(A) = \{0,2,4,5,6,7,8,9,\ldots\} = \N \setminus \{1,3\}.
\]
Since every element of $T_0(A)$ can be written uniquely in the form $m + 5n$ where $m\in \{0,2,4,6,8\}$ and $n\in \N$, we find
\[
\sigma_0(x) = \frac{1 + x^2 + x^4 + x^6 + x^8}{1-x^5}.
\]
On the other hand, note that
\[
T_5(A) = \{0,-3,-5,-6,-8,-9,-10,\ldots\} = -\N \setminus \{-1,-2,-4,-7\},
\]
and it follows that
\[
\sigma_5(x) = \frac{1 + x^{-3} + x^{-6} + x^{-9} + x^{-12}}{1 - x^{-5}}.
\]
Theorem \ref{thm:SumsetGFLocalRepThm} therefore implies that for all $h \geq 6$,
\begin{align*}
\sigma_{hA}(x) &=  \frac{1 + x^2 + x^4 + x^6 + x^8}{1-x^5} + x^{5h} \cdot \frac{1 + x^{-3} + x^{-6} + x^{-9} + x^{-12}}{1 - x^{-5}} \\
&= \frac{1 + x^2 + x^4 + x^6 + x^8 - x^{ 5 h-7} (1 + x^3 + x^6 + x^9 + x^{12})}{1-x^5}
\end{align*}
a compact way to express the elements of $hA$ whenever $h$ is large. (In fact, one can manually check that this identity holds for all $h \geq 3$.)

A key role in the proof of Theorem \ref{thm:SumsetGFLocalRepThm} is played by a generalization of the sumset generating series (\ref{cardGenSeries}).
Any point in $\c_A$ can be written in the form $(a,h)$, where $a\in hA$.
We define the formal power series $\c_A(x,t) \in \Q\llbracket x, t \rrbracket$ by
\begin{equation}
    \c_A(x,t) := \sum_{(a,h) \in \c_A} x^{a} t^h .
\end{equation}
From the definition, we immediately obtain the formula
\begin{equation} \label{eq:ConeStructureSeries}
    \c_A(x,t) = \sum_{h\geq 0} \sigma_{hA}(x) t^h.
\end{equation}
As in the proof of Theorem \ref{thm:SimplexSize}, we will proceed by expressing $\c_A$ as a rational function in $x$ and $t$.
Recall that $\Lambda$ is defined to be the vectors spanned by the lifts of the convex hull of $A$; in this case, we simply have $\Lambda = \{(bn, m+n) : m,n \in \Z\}$, and $\Lambda^+ = \{(bn, m+n) : m,n \in \N\}$.
For brevity we denote the points of $\c_A$ congruent to $(a,1)$ modulo $\Lambda$ by $\s_a$, following our convention from section \ref{sec:ExplicitFormula}.

Recall that one of the technical difficulties in our proof of Theorem \ref{thm:SimplexSize} was the possibility of multiple generators (minimal elements) of $\s_a$. We circumvent this here by introducing the concept of a \emph{virtual generator}, a single point that generates all of $\s_a$ plus possibly a few extraneous points.

\begin{prop}\label{prop:VirtGenerators}
Given $a \in \{0,1,\ldots, b-1\}$, there exists a unique $(g_a,h_a) \in \N^2$ such that $\s_a \subset (g_a,h_a) + \Lambda^+$ and the extraneous set
\[
E_a := \bigg((g_a,h_a) + \Lambda^+\bigg) \setminus \s_a
\]
is finite. We call $(g_a,h_a)$ the \emph{virtual generator} of $\s_a$.
\end{prop}

\noindent
Before proving this, we briefly apply the proposition to the example $A = \{0,1,7,8\}$ from section \ref{sec:ExplicitFormula}.
Recall (see Figure \ref{fig:4elmtex}) that
\[
\s_4 = \bigg((4,4) + \Lambda^+ \bigg) \cup \bigg((28,4) + \Lambda^+ \bigg),
\]
i.e.\ $\s_4$ has two minimal elements $(4,4)$ and $(28,4)$.
Examining Figure \ref{fig:4elmtex}, we see that $(4,1)$ is a virtual generator of $\s_4$, generating all of $\s_4$ plus the extraneous set
\[
E_4 := \bigg((4,1)+ \Lambda^+\bigg) \setminus \s_4
= \left\{(4,1), (4,2), (4,3), (12,2), (12,3), (12,4), (20,3), (20,4), (20,5)  \right\}.
\]
With this intuition in hand, we prove the proposition.

\begin{proof}[Proof of Proposition \ref{prop:VirtGenerators}]
First we prove existence. Note that $(0,0)$ is a virtual generator for $S_0 = \Lambda^+$, so assume $1\leq a < b$. Then it is immediate that $\s_a \subset (a,1) + \Lambda^+$.
Choose maximal $m,n \in \N$ such that
\[
\s_a \subset (a + bn, 1 + m + n) + \Lambda^+,
\]
i.e.\ such that
\(
\s_a \not\subset (a + bk, 1 + j + k) + \Lambda^+
\)
whenever $j > m$ or $k > n$;
such integers are guaranteed to exist since $\s_a$ has finitely many minimal elements.
We claim $(g_a,h_a) = (a + bn, 1 + m + n) $ is a virtual generator of $\s_a$.

Suppose for contradiction that  $E_a$ were infinite. Since $(a,h) + \Lambda^+ \subset \s_a$ whenever $(a,h)\in \s_a$, we must have either $\{(g_a, h_a + n) : n\in \N \} \subset E_a$, in which case
$\s_a \subset (g_a,h_a) + (b,1) + \Lambda^+$,
or $\{(g_a + bn, h_a + n) : n\in \N \} \subset E_a$, in which case $\s_a \subset (g_a,h_a) + (0,1) + \Lambda^+$. Either way we reach a contradiction to the definition of $(g_a,h_a)$. This concludes the proof of existence.

Uniqueness immediately follows because if $(a,h) \neq (b,k)$ then $(a,h) + \Lambda^+$ and $(b,k) + \Lambda^+$ differ by infinitely many points.
\end{proof}

Note that in the example $A = \{0, 1, 7, 8\}$ we considered following Proposition \ref{prop:VirtGenerators}, all the heights appearing in the extraneous set were quite small, as was the height of the virtual generator.
Our previous work implies that this is a general phenomenon.
First, observe that the height of any virtual generator is bounded by the heights of the minimal elements, which we have a bound for thanks to Lemma \ref{lem:minheightbound}:
\begin{equation} \label{eq:BoundOnVirtGen}
h_a \leq b-1 .
\end{equation}
To bound the heights appearing in the extraneous set, note that Proposition \ref{prop:VirtGenerators} implies
\[
\s_a(x,t) = \frac{x^{g_a} t^{h_a}}{(1-t)(1- x^b t)} - \sum_{(n,h) \in E_a} x^n t^h
= \frac{x^{g_a} t^{h_a}-(1-t)(1- x^b t) \sum\limits_{(n,h) \in E_a} x^n t^h}{(1-t)(1- x^b t)} .
\]
Specializing this to $x = 1$ and applying Lemmas \ref{lem:minheightbound} and \ref{lem:genseriesdegbound}, we deduce:

\begin{cor}
For all $(n,h) \in E_a$ we have $h \leq 2b-5$.
\end{cor}

\noindent It follows that
\begin{equation}\label{eq:xGenSeries}
\c_A(x,t)
= Q(x,t) + \frac{1}{(1-t)(1-x^b t)} \sum_{a=0}^{b-1} x^{g_a} t^{h_a}
\end{equation}
where $Q \in \Q[x,t]$ has $t$-degree less than or equal to $2b-5$.

To prove Theorem \ref{thm:SumsetGFLocalRepThm} we need to understand the structure of the tangent cones, which admit a simple expression in terms of the virtual generators of $A$.

\begin{prop}\label{prop:TCGF}
The tangent cone $T_0(A)$ can be written as
\[
T_0(A) =\bigcup_{a = 0}^{b-1} \Big(g_a + b \cdot \N \Big)
\]
where $b \cdot \N = \{0,b,2b,\ldots\}$.
\end{prop}

\begin{proof}
First, observe that $g_a \in T_0(A)$. Indeed, $(g_a,k) \in (g_a,h_a) + \Lambda^+$ for all $k \geq h_a$, and only finitely many of these can live outside of $\s_a$; it follows that $\s_a$, and hence $\c_a$, must contain a point of the form $(g_a,h)$ for some $h \in \N$.
By construction, $g_a \equiv a \mod b$.
Thus the claim boils down to showing that $g_a$ is the \emph{smallest} element of $T_0(A)$ congruent to $a \mod{b}$.

Pick any $m \in \{n\in T_0(A) : n \equiv a \mod{b}\}$.
Since $m \in T_0(A)$, we deduce $(m,h) \in \c_A$ for some $h \in \N$, and $m \equiv a \mod b$ then implies that $(m,h) \in \s_a$. By definition, $\s_a \subset (g_a,h_a) + \Lambda^+$, so $m \geq g_a$.
\end{proof}

\noindent Furthermore, the virtual generators possess a similar symmetry to the minimal elements of $\c_A$.

\begin{prop}\label{prop:VGSym}
Given $0 \leq a < b$, let $(g_a,h_a)$ denote the corresponding virtual generator of $A$ and $(g_a',h_a')$ denote the virtual generator of $b - A$.
Then $h_a = h_{b-a}'$ and $g_{b-a}' = g_a - b h_a$

\end{prop}
\begin{proof}
Write
\[
E_a = \left\{(g_a,h_a) + m (0,1) + n(b,1) : (m,n) \in I_a \right\}
\]
where $I_a \subset \N\times \N$ is some finite set.
If we denote the elements of $\c_{A-b}$ congruent to $b-a \mod{b}$ by $\s_{b-a}'$, then since $h(b - A) = hb - hA$,
\[
\s_{b-a}' = \bigg((g_a - b h_a , h_a) + \Span_\N\{(0,1),(-b,1)\}\bigg) \setminus E_{b-a}'
\]
where
\[
E_{b-a}' = \left\{(g_a - b h_a ,h_a) + m (0,1) + n(-b,1) : (m,n) \in I_a \right\}.
\]
The claim now follows by uniqueness of virtual generators.
\end{proof}

\begin{proof}[Proof of Theorem \ref{thm:SumsetGFLocalRepThm}]
By Propositions \ref{prop:TCGF} and \ref{prop:VGSym}, we obtain the following expressions for the power series corresponding to the tangent cones $T_0(A)$ and $T_b(A)$:
\[
\sigma_0(x) = \frac{1}{1 - x^b}\sum_{a = 0}^{b-1} x^{g_a}
\qquad \qquad \text{and} \qquad \qquad
\sigma_b(x) = \frac{1}{1 - x^{-b}}\sum_{a = 0}^{b-1} x^{g_a - b h_a}
\]
where $(g_a,h_a)$ are the virtual generators corresponding to each $0\leq a < b$.
From (\ref{eq:xGenSeries}) we know the existence of some $Q\in \Q[x,t]$ of $t$-degree less than or equal to $2b-5$ such that
\begin{align*}
    \c_A(x,t) &= Q(x,t) + \frac{1}{(1-t)(1-x^b t)} \sum_{a=0}^{b-1} x^{g_a} t^{h_a} \\
    &=  Q(x,t) + \left(\sum_{a=0}^{b-1} x^{g_a} t^{h_a} \right)\left(\sum_{h\geq 0} t^h \right)  \left(\sum_{j\geq 0} x^{jb} t^j \right) \\
    &= Q(x,t) + \left(\sum_{a=0}^{b-1} x^{g_a} t^{h_a} \right)\sum_{h\geq 0} \left(\sum_{j = 0}^h x^{j b}\right) t^h .
\end{align*}
We can rearrange the product of sums above as follows:
\begin{align*}
    \sum_{h\geq 0} \Bigg[  \frac{1}{1-x^b}\sum_{a=0}^{b-1} x^{g_a} t^{h_a} &+ \frac{x^{hb}}{1-x^{-b}}\sum_{a=0}^{b-1} x^{g_a} t^{h_a} \Bigg] t^h  = \\
    \sum_{h\geq 0}\left[ \frac{1}{1-x^b}  \sum_{a=0}^{b-1} x^{g_a} t^{h_a} \right]t^h &+ \sum_{h\geq 0} \left[ \frac{x^{hb}}{1-x^{-b}} \sum_{a=0}^{b-1} x^{g_a} t^{h_a}  \right] t^h.
\end{align*}
Setting $H = \max\{h_0, \ldots, h_{b-1}\}$, we may further rewrite this in the form
\[
P(x,t) + \sum_{h\geq H}\left[ \frac{1}{1-x^b}  \sum_{a=0}^{b-1} x^{g_a} \right]t^h + \sum_{h\geq H} \left[ \frac{x^{hb}}{1-x^{-b}} \sum_{a=0}^{b-1} x^{g_a - bh_a} \right] t^h =
P(x,t) + \sum_{h\geq H} (\sigma_0(x) + x^{hb}\sigma_b(x))t^h ,
\]
where $P(x,t) \in (\Q(x))[t]$ has $t$-degree strictly less than $H$.
Putting this all together, we've shown that
\[
\c_A(x,t) = P(x,t) + Q(x,t) + \sum_{h\geq H} (\sigma_0(x) + x^{hb}\sigma_b(x))t^h .
\]
On the other hand, recall from \eqref{eq:ConeStructureSeries} that
\[
\c_A(x,t) = \sum_{h\geq 0} \sigma_{hA}(x) t^h .
\]
It follows that $\sigma_{hA}(x) = \sigma_0(x) + x^{hb}\sigma_b(x)$ for $h > \max\{\deg_t Q, H-1\}$. Since $\deg_t Q \leq 2b-5$ by \eqref{eq:xGenSeries} and $H \leq b-1$ by \eqref{eq:BoundOnVirtGen},
we conclude that
\[
\phantom{\qquad \forall h \geq \max\{2b-4, b-1\} .}
\sigma_{hA}(x) = \sigma_0(x) + x^{hb}\sigma_b(x)
\qquad \forall h \geq \max\{2b-4, b-1\} .
\]
If $b \geq 3$ this implies the claim.
The only remaining case is $b=1$ and $b=2$, i.e.\ $A = \{0,1\}$ and $A = \{0,1,2\}$, respectively.
In either case, (\ref{eq:sumsetBrion}) trivially holds for all $h \in \N$.
\end{proof}

\section{Explicit Khovanskii for arbitrary sumsets: Proof of Theorem \ref{thm:dplus2}} \label{sec:nonsimplex}

%Thus far, we've only considered sets whose convex hulls are simplexes. The goal of this section is to demonstrate that our methods can be pushed to handle more general sets. Specifically, we will prove Theorem \ref{thm:dplus2}, which we previously proved under the additional assumption that the convex hull of $A$ was a simplex.
Recall that in section \ref{sec:ExplicitFormula} we proved Theorem \ref{thm:dplus2} under the additional hypothesis that the convex hull of $A$ is a simplex. The goal of this section is to prove Theorem \ref{thm:dplus2} in full generality.

Consider a given $A \subset \Z^d$ consisting of precisely $d+2$ points. As usual, we may assume that $A$ contains $0$ and that $A$ generates $\Z^d$ additively. Denote the nonzero elements of $A$ by $\bm{v}_0, \ldots, \bm{v}_d$; without loss of generality, we may assume
\(
\bm{v}_1, \ldots, \bm{v}_d
\)
are linearly independent.
Our starting point is the observation that
\begin{equation}\label{eq:1-2StartingPoint}
hA =
\bigcup_{j = 0}^h \left(j \bm{v}_0 + C_{h-j} \right) ,
\quad \text{where} \quad
C_k := \left\{\sum_{i = 1}^d n_i \bm{v}_i : n_i \in \N  \text{ and } \sum_{i = 1}^d n_i  \leq k \right\} .
\end{equation}
We're naturally led to study truncated cones of the form
\[
A_{j,h} := j \bm{v}_0 + C_{h-j} .
\]
When $h$ is small enough, all these truncated cones are disjoint, in which case we can compute $|hA|$ easily.
Set
\[
H :=
\max \{
\ell : A_{j,h} \cap A_{j',h} = \varnothing
\text{ whenever }
j \neq j' \text{ and } h < \ell
\} ;
\]
this is necessarily finite, since the $d + 1$ nonzero elements of $A$ are linearly dependent.
Thus when $h < H$,
\[
|hA|
= \sum_{j=0}^{h} |A_{j,h}|
= \sum_{j=0}^{h} |C_{h-j}|
= \sum_{j = 0}^h \binom{h - j + d}{d}
= \binom{h + d + 1}{d + 1}.
\]
This agrees with the first part of Theorem \ref{thm:dplus2}. Our task now is to compute the value of $H$, and to explore what happens when $h \geq H$, i.e.\ once the truncated cones $A_{j,h}$ intersect one another.
We are able to give an explicit description of such intersections, but to do so we require a bit more notation.

Let $B := \{\bm{v}_1, \bm{v}_2, \ldots, \bm{v}_d\}$, and consider the cone
\(
\Gamma := \textup{span}_\N \ B .
\)
%Thus, for example, the truncated cone $C_m$ can be described $C_m = \{ \bm{v} \in \Gamma : \|\bm{v}\|_1 \leq m\}$, where $\|\bm{v}\|_1$ is the $L^1$ norm of $\bm{v}$ with respect to the basis $B$.
Since $B$ is a basis of $\R^d$, the $\Z$-span of $B$ has finite index in $\Z^d$, whence $\bm{v}_0$ has finite order in $\Z^d / \textup{span}_\Z \ B$.
Let $N$ denote the order of $\bm{v}_0$; in particular, $N \bm{v}_0$ is an element of the lattice $\textup{span}_\Z \ B$.
Finally, define $\bm{w} \in \Gamma$ via the relation
\[
\Gamma \cap (N \bm{v}_0 + \Gamma)
= \bm{w} + \Gamma .
\]
We can now state the promised explicit description of the intersections of truncated cones $A_{j,h}$.
First, observe that $C_k \subset \Gamma$ for any $k$; it immediately follows that if $A_{j,h}$ and $A_{j',h}$ intersect, then $j \equiv j' \mod N$.
Moreover, it turns out any intersection boils down to a single intersection:

\begin{lemma}\label{lem:MultIntersectOfAs}
Suppose all elements of $I \subseteq [0,h]$ are congruent (mod $N$). Then
\(
\bigcap\limits_{j \in I} A_{j,h}
= A_{m,h} \cap A_{M,h} ,
\)
where $m := \min \ I$ and $M := \max \ I$.
\end{lemma}
\noindent
It therefore suffices to describe the intersection of two cones:

\begin{lemma} \label{lem:IntersectOfAs}
Let $a,j,h \in \N$.
Then
\(
A_{j,h} \cap A_{j+aN,h}
= j \bm{v}_0 + a \bm{w} + C_{h - j - a H} ,
\)
where $C_k := \varnothing$ for $k < 0$.
\end{lemma}

\noindent
We will prove Lemmas \ref{lem:MultIntersectOfAs}   and \ref{lem:IntersectOfAs} below, using the same circle of ideas we developed in Sections \ref{sec:ExplicitFormula}--\ref{sec:BrionFormula}.
But first, we demonstrate their utility by giving a short derivation of Theorem \ref{thm:dplus2} from them.

Recall from \eqref{eq:1-2StartingPoint} that we can express $hA$ as a union of truncated cones.
Breaking this up further (mod $N$), we have
\begin{equation}\label{eq:hADiscreteUnion}
hA =
\bigsqcup_{j = 0}^{N-1}
\bigcup_{a \geq 0} A_{j+aN ,h} .
\end{equation}
Inclusion-exclusion implies
\begin{equation}  \label{eq:truncIncExc}
\left|\bigcup_{a \geq 0} A_{j+aN ,h}\right| =
\sum_{I \subseteq \left[\frac{h-j}{N}\right]} (-1)^{|I|+1} \Bigg|\bigcap_{a \in I} A_{j+aN,h}\Bigg|
\end{equation}
where $[\a] := \{n \in \N : n \leq \a\}$.
Lemma \ref{lem:MultIntersectOfAs} allows us to simplify the intersection on the right hand side, whence by Lemma \ref{lem:IntersectOfAs} we find
\begin{equation}\label{eq:truncIntSize}
\Bigg|\bigcap_{a \in I} A_{j+aN,h}  \Bigg|
= |A_{j+mN,h} \cap A_{j+MN,h}|
= |C_{h - (j+mN) - (M - m) H}|
= \binom{h - (j+mN) - (M - m) H + d}{d}
\end{equation}
where $m := \min \ I$ and $M := \max \ I$.

We now claim that the only subsets $I$ that do not cancel out in (\ref{eq:truncIncExc}) are the singleton sets $\{k\}$ and sets of the form $\{k, k+1\}$.
To see this, fix $m, M$ with $M \geq m+2$.
All the $I$ in (\ref{eq:truncIncExc}) that have minimal element $m$ and maximal element $M$ contribute
\[
\sum_{\substack{I \subseteq \{m, m + 1, \ldots , M\} \\ m, M \in I }} (-1)^{|I|+1} \Bigg|\bigcap_{a \in I} A_{j+aN,h}\Bigg|
= \sum_{\substack{I \subseteq \{m, m + 1, \ldots , M\} \\ m, M \in I }} (-1)^{|I|+1} \binom{h - (j+mN) - (M - m) H + d}{d} .
\]
The key observation is that the summands on the right hand side do not depend on $I$, but only on its cardinality. Ordering the sum by size of $I$, we find
\[
\begin{split}
\sum_{\substack{I \subseteq \{m, m + 1, \ldots , M\} \\ m, M \in I }} & (-1)^{|I|+1} \Bigg|\bigcap_{j \in I} A_{j,h}\Bigg|
= \\
&= \binom{h - (j+mN) - (M - m) H +d}{d}
\sum_{\ell = 2}^{M-m+1} (-1)^{\ell+1}
\#\Big\{I \subseteq \{m,\ldots, M\} : m,M \in I \text{ and } |I| = \ell \Big\} \\
&= \binom{h - (j+mN) - (M - m) H +d}{d}
\sum_{\ell = 2}^{M-m+1} (-1)^{\ell+1}
\binom{M-m-1}{\ell-2} = 0 .
\end{split}
\]
Thus, the only $I$ that contribute to \eqref{eq:truncIncExc} are singletons or pairs of consecutive integers, as claimed.
Combining this with \eqref{eq:hADiscreteUnion} and \eqref{eq:truncIncExc} yields
\[
\begin{split}
|hA|
&= \sum_{j = 0}^{N-1}
\sum_{a \in \left[\frac{h-j}{N}\right]}
\bigg( |A_{j+aN,h}| - |A_{j+aN,h} \cap A_{j+(a+1)N,h}| \Bigg)
= \sum_{k=0}^{h} \bigg( |C_{h-k}| - |C_{h-k-H}| \Bigg) \\
&= \sum_{k=0}^{h} |C_{h-k}| -
\sum_{k=H}^{h+H} |C_{h-k}|
= \sum_{k=0}^{h} \binom{h-k+d}{d} -
\sum_{k=H}^{h} \binom{h-k+d}{d} \\
&= \binom{h+d+1}{d+1} - \binom{h-H+d+1}{d+1}
\end{split}
\]
for any $h \geq H$.

All that remains is to compute $H$. We can do this easily by using Khovanskii's theorem:
for any $h \geq H$,
\[
\binom{h + d + 1}{d + 1} - \binom{h-H+d+1}{d+1}
\]
is a polynomial of degree $d$ and leading coefficient $H/d!$, so Theorem \ref{thm:KhoMainThm} implies
\[
H = \vol(\Delta_A) \cdot d!  .
\]
This concludes the proof of Theorem \ref{thm:dplus2} assuming Lemmas \ref{lem:MultIntersectOfAs} and \ref{lem:IntersectOfAs}. We now circle back and prove these.

As in our work in
Sections \ref{sec:ExplicitFormula}--\ref{sec:BrionFormula}, we disentangle the geometry from the combinatorics by lifting $A$ and its associated truncated cones $A_{j,h}$ to one dimension higher.
More precisely, recall that the \emph{lift} of $\bm{v} \in \Z^{d}$ is the vector $\widetilde{\bm{v}} := (\bm{v},1) \in \Z^{d+1}$.
Then
\begin{equation}\label{eq:LiftOfTruncCones}
A_{j,h} =
\pi\Big( \widetilde{A}_j \cap \h_h \Big)
\end{equation}
where
\[
\widetilde{A}_j :=
j \widetilde{\bm{v}}_0 + \widetilde{\Gamma}
\quad , \quad
\widetilde{\Gamma} := \Span_\N\{\widetilde{\bm{0}}, \widetilde{\bm{v}}_1, \ldots, \widetilde{\bm{v}}_d\}
\quad , \quad
\h_{h} := \left\{\bm{x} \in \R^{d+1} : \text{height} (\bm{x}) = h \right\} ,
\]
and $\pi : \R^{d+1} \to \R^d$ is the projection onto the first $d$ coordinates.
Thus $\widetilde{\Gamma} \subseteq \Z^{d+1}$ is an infinite cone, $\widetilde{A}_j$ is a translation of this cone, and we are viewing $A_{j,h}$ as a level set of this translated cone.
%Note that the difference set $\widetilde{\Gamma} - \widetilde{\Gamma}$ is the lattice generated by the same set of vectors over $\Z$.

\begin{proof}[Proof of Lemma \ref{lem:MultIntersectOfAs}]
By \eqref{eq:LiftOfTruncCones}, it suffices to prove
\begin{equation}\label{eq:liftedMultInt}
\underbrace{\bigcap_{j\in I}  \widetilde{A}_j}_L
= \underbrace{\widetilde{A}_m \cap \widetilde{A}_M}_R .
\end{equation}
It is clear that $L \subseteq R$, so we focus on the reverse inclusion.
Fix any $j \in [m,M]$ congruent to $m \mod N$; our goal is to show that
\(
\widetilde{A}_j \supseteq R .
\)
Because $R$ is an intersection of cones we may write $R = \bm{x} + \widetilde{\Gamma}$ for some $\bm{x} \in \Z^{d+1}$.
In particular, there exist $\bm{y} , \bm{z} \in \widetilde{\Gamma}$ such that
\[
\bm{x}
= m\widetilde{\bm{v}}_0 + \bm{y}
= M\widetilde{\bm{v}}_0 + \bm{z} .
\]
It follows that
\begin{equation}\label{eq:DiffIsInLattice}
\bm{x} - j \widetilde{\bm{v}}_0
= (M-j) \widetilde{\bm{v}}_0 + \bm{z}
\in \widetilde{\Gamma} - \widetilde{\Gamma} ,
\end{equation}
since $(M-j) \widetilde{\bm{v}}_0$ is a positive integer multiple of $N \widetilde{\bm{v}}_0$, which lives in the lattice $\widetilde{\Gamma} - \widetilde{\Gamma}$ by definition of $N$.

Next, write $j =  m + t (M-m)$ for some $t\in [0,1]$. It follows that
\[
\bm{x} - j \widetilde{\bm{v}}_0
= (1-t) \bm{y} + t \bm{z}
\in \Delta_{\widetilde{\Gamma}} ,
\]
the convex hull of $\widetilde{\Gamma}$.
Combining this with \eqref{eq:DiffIsInLattice}, we deduce that $\bm{x} - j \widetilde{\bm{v}}_0 \in \widetilde{\Gamma}$, whence
\(
R = \bm{x} + \widetilde{\Gamma} \subseteq
j \widetilde{\bm{v}}_0 + \widetilde{\Gamma}
= \widetilde{A}_j.
\)

\end{proof}

\begin{proof}[Proof of Lemma \ref{lem:IntersectOfAs}]
Recall that
\(
\Gamma \cap (N \bm{v}_0 + \Gamma)
= \bm{w} + \Gamma .
\)
It immediately follows that
\[
\widetilde{A}_0 \cap \widetilde{A}_N = (\bm{w},\eta) + \widetilde{\Gamma}
\]
for some $\eta$.
This implies
\[
\widetilde{A}_{kN} \cap \widetilde{A}_{(k+1)N}
= (kN \widetilde{v}_0 + \widetilde{A}_{0}) \cap (kN \widetilde{v}_0 + \widetilde{A}_{N})
= (\bm{w},\eta) + \widetilde{A}_{kN} .
\]
From this we deduce
\[
\begin{split}
\widetilde{A}_{kN} \cap \widetilde{A}_{(k+1)N} \cap \widetilde{A}_{(k+2)N}
&= \left(\widetilde{A}_{kN} \cap \widetilde{A}_{(k+1)N}  \right) \cap \left( \widetilde{A}_{(k+1)N} \cap \widetilde{A}_{(k+2)N}  \right)  \\
&= \left((\bm{w},\eta)  + \widetilde{A}_{kN} \right) \cap \left((\bm{w},\eta)  + \widetilde{A}_{(k+1)N} \right) \\
&= 2(\bm{w},\eta) + \widetilde{A}_{kN} .
\end{split}
\]
Proceeding by induction (and using Lemma \ref{lem:MultIntersectOfAs}) we conclude
\[
\widetilde{A}_0 \cap \widetilde{A}_{aN}
= \bigcap_{k=0}^a \widetilde{A}_{kN}
= a(\bm{w},\eta) + \widetilde{\Gamma} .
\]
Intersecting the left and right hand sides with $\h_h$ yields
\[
A_{0,h} \cap A_{aN,h}
=  a \bm{w} + C_{h - a \eta} ,
\]
whence
\begin{equation}\label{eq:LemmaWithEta}
A_{j,h} \cap A_{j+aN,h}
= j \bm{v}_0 + a \bm{w} + C_{h - j - a \eta} .
\end{equation}
To finish the proof, all that remains is to show that $\eta = H$.

First, observe that \eqref{eq:LemmaWithEta} yields $A_{0,\eta} \cap A_{N,\eta} = \{\bm{w}\}$, so $\eta \geq H$.
On the other hand, by definition of $H$ there exist $a \geq 1$ and $j \geq 0$ such that
\(
A_{j,H} \cap A_{j+aN,H} \neq \varnothing .
\)
Our identity \eqref{eq:LemmaWithEta} then implies that $H - j - a \eta \geq 0$, whence $\eta \leq H$.
\end{proof}

\section{Conclusions: recap, conjectures, and observations} \label{sec:FurtherDirections}

Recall that Khovanskii's theorem asserts that for any $A \subset \Z^d$ there exists some integer $H$ (which we called the \emph{phase transition}) such that the cardinality of $hA$ is given by some polynomial in $h$ for all $h \geq H$. Our work improved on this in several ways:
\begin{itemize}
\item in Theorem \ref{thm:SimplexSize} we gave an explicit upper bound on the phase transition when $\Delta_A$ is a simplex,
\item in Theorem \ref{thm:SimplexStructure} we obtained an analogous result on the structure of $hA$, again with an explicit upper bound on the phase transition when $\Delta_A$ is a simplex,
\item in Theorem \ref{thm:SumsetGFLocalRepThm} we demonstrated that for any $A \subset \Z$ one can give a compact expression capturing the structure of $hA$ for all sufficiently large $h$, and gave an explicit upper bound on the phase transition, and
\item in Theorem \ref{thm:dplus2} we gave a complete description of $hA$ for all $h$ in the case that $A$ is small.
\end{itemize}
All but the last of these offer room for improvement. The goal of this section is to make some conjectures and share some curious empirical observations.

In section \ref{sec:BrionFormula}, we proved Theorem \ref{thm:SumsetGFLocalRepThm}, a Brion-type formula in dimension 1.
We expect that one can generalize this theorem to higher dimensions by associating to the point $\bm{a} = (a_1,\ldots, a_d) \in \Z^d$ the monomial weight
\[
\bm{x}^{\bm{a}} := x_1^{a_1} \cdots x_d^{a_d}.
\]
Defining the generating functions (\ref{eq:sumsetXGF}) and (\ref{eq:TCXGF}) with $\bm{x}^{\bm{a}}$ in place of $x^a$, we expect the following analogue of (\ref{eq:sumsetBrion}) to hold for all sufficiently large $h$:
\[
\sigma_{hA}(\bm{x}) = \sum_{i = 1}^{d+1} \bm{x}^{\bm{h \bm{v}_i}} \sigma_{\bm{v}_i}(\bm{x}) .
\]
(This can also be thought of as a generating function analogue of Theorem \ref{thm:SimplexStructure}.)
We conjecture that this formula is valid whenever
\begin{equation}\label{eq:SLConj}
    h \geq \vol(\Delta_A) \cdot d! - |A| + 2 .
\end{equation}
One reason we did not pursue this theorem in the general case is an additional technical difficulty:
in dimension 1 the extraneous sets are  finite collections of points, but
in higher dimensions they are instead finite unions of hypersurfaces.
We invite the motivated reader to carry out this strategy and obtain a general version of Theorem \ref{thm:SumsetGFLocalRepThm}.

Our bound on the phase transition in Khovanskii's Theorem is likely not optimal.
Over $\Z$ it is known that Theorem \ref{thm:SimplexStructure} holds for $h\geq b - |A| + 2$, thanks to work of Granville and Walker \cite{GranvilleWalker}.
For higher dimensions, we conjecture that Theorem \ref{thm:SimplexSize}  holds under the assumption (\ref{eq:SLConj}) and that Theorem \ref{thm:SimplexStructure} holds under the assumption that
\[
h \geq \vol(\Delta_A) \cdot d! - |A| + d + 1
\]
without any assumption on the convex hull of $A$; note that this specializes to Granville and Walker's bound in the case $d=1$.
The reasoning behind our conjecture is that our proof shows that the phase transition measures how long it takes lifts of elements of $A$ to fill in all of the residue classes of $\c_A$.
Thus the case of $A$ containing $d + 2$ elements should take the longest time to fill in all residue classes, so we expect this to be the worst case scenario.
Furthermore, the more elements $A$ contains, the faster it should fill up all of the residue classes, so the phase transition should occur earlier the larger $|A|$ is.
The linear decrease with respect to $|A|$ is motivated by Granville and Walker's result \cite{GranvilleWalker} over $\Z$.
%Moreover, this conjecture matches our exact formula for sets of size $d + 2$.

Our conjecture on the phase transition is borne out by computations, but in sometimes unexpected ways.
For example, for the set
\(
A = \{(0,0),(-1,1),(1,2),(4,0)\}
\)
one can show that
\[
\c_A(t) = \frac{1 - t^{11}}{(1-t)^4} ,
\]
despite the presence of minimal elements of $\c_A$ with heights as large as 14.
The fact that the degree of $\c_A(t)$ is smaller than 14 comes from a seemingly miraculous cancellation that occurs when adding together the generating series of the sets $\s_{\bm{\pi}}$.
Perhaps even more surprising is that this miraculous cancellation persists even when computing the structural generating functions as in Theorem \ref{thm:SumsetGFLocalRepThm}. For example, using the same set $A$ as above but
keeping track of the positions in $\c_A$ with weights $x$ and $y$, one finds that
\begin{equation}\label{eq:xyGF1}
\c_A(x,y,t) = \frac{1 -  x^4 y^8 t^{11}}{(1 - t) (1 - x^4 t) (1 - x^{-1} y t) (1 - x y^2 t)}.
\end{equation}

We conclude our discussion with some tantalizing numerology. Consider the set
\[
B = \{(0,0),(1,2),(2,1),(3,1)\} ,
\]
whose convex hull is a simplex. The method from section \ref{sec:ExplicitFormula} produces
\begin{equation}\label{eq:xyGF2}
    \c_B(x,y,t) = \frac{1 - x^{10} y^5 t^5}{(1-t)(1-xy^2 t) (1 - x^3y t)(1 - x^2 y t)}.
\end{equation}
The remarkably similar form of (\ref{eq:xyGF1}) and (\ref{eq:xyGF2}) suggests that there may be a unified approach to proving Theorem \ref{thm:dplus2} without treating the simplicial and non-simplicial cases separately.
The  $x^{10} y^5 t^5$ term of the numerator of (\ref{eq:xyGF2}) admits a nice interpretation:
it is the weight assigned to the lift of the interior point $(2,1)$ of $B$ raised to the power 5, the volume of the fundamental domain of $\c_B$.
Unfortunately, the corresponding term $x^4 y^8 t^{11}$ in (\ref{eq:xyGF1}) does not seem to have such a nice interpretation. The volume of the fundamental domain of $\c_A$ is 11, so $x^4 y^8 t^{11}$ would correspond to the weight of the lift of the point $(4/11, 8/11)$, which does not lie in $A$, and moreover isn't even a lattice point!
A proper interpretation of the term $ x^4 y^8 t^{11}$ appearing in (\ref{eq:xyGF1}) may well be the key to extending our results to arbitrary $A \subset \Z^d$.

%%% AUTHOR: body of paper starts here
%\section{Introduction}
 %The body of your paper goes here~\cite{cilleruelo}.

%\newpage %% AUTHOR: please comment out this line.  It serves only
%%   to demonstrate both types of header line in daj-template.pdf

%\section{Expansion estimates}

 %More of the body of your paper goes here~\cite{bergelson-johnson-moreira}.

%%% AUTHOR: optional appendix here
%\appendix %% you may comment this out if no Appendix
%\section*{Appendix}
%\section{Improving the constants}
%Material is placed here as needed.

%%% AUTHOR: optional acknowledgments here
\section*{Acknowledgments} %%  you may comment this out if no Ackno
We're grateful to Andrew Granville and Aled Walker for sharing their work with us, as well as for pointing out a subtle difficulty in our initial approach. We'd also like to thank Ben Logsdon and Ralph Morrison for providing helpful feedback on early versions of this paper,
Ilija Vre\'{c}ica \cite{Vrecica} for discovering an error in our original proof of Theorem \ref{thm:dplus2},
and the anonymous referees for their meticulous work---their comments improved the clarity of the manuscript.

%%% AUTHOR:
%%% Bibliography goes here. Note that the arXiv cannot process bibtex
%%% or biber bibliographies.  Example of acceptable bibliograpy format:

%% AUTHOR: You can generate such a bibliography from a .bib file by
%% running pdflatex/bibtex/pdflatex/pdflatex and then pasting the .bbl file
%% between \begin{thebibliography} and \end{bibliography}

%%% AUTHOR: Include a short description of each author following the
%%% structure below. Use the same short tags used previously.
%%% Use \imageat{} and \imagedot{} instead of "@" and "." in
%%% email addresses-this replaces the symbols with graphics to avoid
%%% e-mail address harvesting from the .pdf file
\begin{dajauthors}
\begin{authorinfo}[mjc]
  Michael J. Curran\\
  University of Oxford\\
  Oxford, United Kingdom\\
  Michael\imagedot{}Curran\imageat{}maths\imagedot{}ox\imagedot{}ac\imagedot{}uk \\
  %\url{https://www.cs.elte.hu/erdos}
\end{authorinfo}
\begin{authorinfo}[lg5]
  Leo Goldmakher\\
  Williams College\\
  Williamstown, MA, USA.\\
  Leo\imagedot{}Goldmakher\imageat{}williams\imagedot{}edu \\
  %\url{http://www.csc.kth.se/~johanh}
\end{authorinfo}
\end{dajauthors}

\end{document}